\documentclass{article}
\usepackage{amssymb}
\usepackage{amsmath}
\usepackage{amsfonts}
\usepackage{bbding}
\usepackage{amsfonts}
\usepackage{graphicx}

\catcode`ö=\active
\defö{\"{o}}
\newtheorem{theorem}{Theorem}[section]

\newtheorem{corollary}[theorem]{Corollary}

\newtheorem{definition}[theorem]{Definition}

\newtheorem{lemma}[theorem]{Lemma}

\newtheorem{proposition}[theorem]{Proposition}

\newenvironment{proof}[1][Proof]{\noindent\textbf{#1.} }{\ \rule{0.5em}{0.5em}}
\input{tcilatex}

\begin{document}

\begin{center}
\ \textbf{\large Solvability of Nonlinear Elliptic Type Equation
With Two Unrelated Non standard Growths}

\bigskip

U\u{g}ur Sert and Kamal Soltanov \footnotetext{%
U. Sert(\Envelope): {}Faculty of Science, Department of Mathematics,
Hacettepe University, 06800, Beytepe, Ankara, Turkey. e-mail:
usert@hacettepe.edu.tr
\par
K. Soltanov: Faculty of Science, Department of Mathematics, Hacettepe
University, 06800, Beytepe, Ankara, Turkey. e-mail: soltanov@hacettepe.edu.tr%
}
\end{center}

\noindent \textbf{Abstract}. In this paper, we study the
solvability of the nonlinear Dirichlet problem with sum of the
operators of independent non standard growths
\begin{equation*}
-div\left(\left\vert\nabla u\right\vert^{p_{1}\left(
x\right)-2}\nabla u\right)-\sum\limits_{i=1}^{n}D_{i}\left(
\left\vert u\right\vert ^{p_{0}\left(
x\right)-2}D_{i}u\right)+c\left( x,u\right) =h\left( x\right)
,\text{ \ \ }x\in \Omega
\end{equation*}%
in a bounded domain $\Omega \subset \mathbb{R}^{n}$. Here, one of
the operators in the sum is monotone and the other is weakly
compact. We obtain sufficient conditions and show the existence of
weak solutions of the considered problem by using monotonicity and
compactness methods together.

\bigskip


\bigskip


\section{Introduction}

In this work, we investigate the Dirichlet problem for the nonlinear
elliptic equation with variable nonlinearity
\begin{equation*}
\left\{
\begin{array}{l}
-div\left(\left\vert\nabla u\right\vert^{p_{1}\left(
x\right)-2}\nabla u\right)-\sum\limits_{i=1}^{n}D_{i}\left(
\left\vert u\right\vert ^{p_{0}\left(
x\right)-2}D_{i}u\right)+c\left( x,u\right) =h\left( x\right) \\
u\mid _{\partial \Omega }=0%
\end{array}%
\right.  \tag{1.1}
\end{equation*}%
where $x\in\Omega \subset \mathbb{R}^{n}\left( n\geq 3\right) $ is
a bounded domain which has sufficiently smooth boundary (at least
Lipschitz boundary), $D_{i}\equiv\partial/\partial x_{i}$,
$p_{0},$ $p_{1\text{ }}$are nonnegative measurable functions
defined on $\Omega ,$ $h$ is a generalized function and $c:\Omega
\times
\mathbb{R}
\rightarrow
\mathbb{R}
$, $c\left( x,\tau \right) $ is a function with variable nonlinearity in $%
\tau $ (for example, $%
c\left( x,u\right) =c_{0}\left( x\right) \left\vert u\right\vert ^{\alpha
\left( x\right) -2}u+c_{1}\left( x\right) $, see Section 2).

We denote the operators $A$ and$\ B$ with \newline $A\left(
u\right) :=-div\left(\left\vert\nabla u\right\vert^{p_{1}\left(
x\right)-2}\nabla u\right),$ $B\left( u\right)
:=-\sum\limits_{i=1}^{n}D_{i}\left( \left\vert u\right\vert
^{p_{0}\left( x\right) -2}D_{i}u\right) +c\left( x,u\right). $

There has been recently a considerable interest in the study of
equations and variational problems with variable exponents of
nonlinearities due to their applications. Nonlinear equations
including the operator $A\left( u\right) $ is a rather common
nonlinear problem with variable exponent which is known as
$p_{1}\left( .\right) $-Laplacian equation. This kind of problems
have been studied in various contexts by many authors [1, 3, 4, 6,
27] and have wide range of application areas in the mathematical
modeling of non-Newtonian fluids [7, 18, 19], theory of elasticity
and hydrodynamics [28], thermistor problem [26] and in image
restoration [9] etc. (Indeed, all application areas mentioned
above are also valid for problem (1.1)). On the other hand, the
equations of the type $B\left( u\right) =h$ are rarely researched.
For instance in [5], a similar type of problem for $B\left(
u\right) =h$ was studied by Antontsev and Shmarev who investigated
the regularized problem to show the existence of weak solution.
Such equations may appear in the mathematical modeling of the
process of nonstable filtration of an ideal barotropic gas in a
nonhomogeneous porous medium (for sample see [3]). We also refer
[14, 17, 19] for the several of the most important applications of
(1.1) and nonlinear partial differential equations with variable
exponent arise from mathematical modeling of suitable processes in
mechanics, mathematical physics, image processing etc.

To the best of our knowledge, by now there has not been any
studies on the existence of solutions for the elliptic equations
of the type (1.1) with variable exponents of nonlinearity.
However, we note that a similar problem to (1.1) with constant
exponents was investigated in [24]. More exactly, in [24] the
question of the solvability of an operator equation when the case
the operator is in the form of sum of a weakly compact and
pseudo-monotone operators was answered. In the present paper, we
study the similar type of operator equation in a model problem
when the operators in the sum with variable nonlinearity and
obtain the sufficient conditions for solvability. The our goal of
studying the model problem is to provide a more understandable and
explicit way for the established results in this article.

The main feature of the equation $A\left( u\right) +B\left(
u\right) =h$ is that, the exponents $p_{0}\left( x\right) $ and
$p_{1}\left( x\right) $ are independent of each other. Thus,
neither $A$ nor $B $ is the main part of this equation. It is need
to note that if $A$ is the main part of the equation, i.e the
exponents are dependent each other, the results for the theory of
pseudo-monotone operators can be used to investigate the problem.
However, in the case that we consider, any methods which is merely
related to monotonicity can not be used.

We use the basic general solvability theorem [24], (Theorem 2.5) to prove
the existence of weak solution of the problem (1.1). In order to apply this
theorem to existence theorem (Theorem 2.4) for problem (1.1), we obtain
sufficient conditions and prove the monotonicity of the operator $A$ and
weak compactness of $B$ on proper spaces under these conditions, and then we
get the solvability of posed problem by simultaneously using monotonicity
and compactness.

This paper is organized as follows: In the next section, we
present the assumptions, definition of the weak solution and
description of the main result. For this purpose, we also define
some function classes which are required to study the posed
problem. In Section 3, firstly, we establish some integral
inequalities to investigate the function classes (pseudo-norm
spaces) defined in previous section and afterwards verify some
necessary lemmas and theorems which indicate the relation of these
spaces with the Lebesgue and Sobolev spaces with variable exponent
and the continuous and compact embeddings of these function spaces
etc. In Section 4, we give the proof of the main theorem (Theorem
2.4) of this paper by the help of the embedding results obtained
in Section 3.

\section{Statement of The Problem and The Main Result}

Let $\Omega \subset \mathbb{R} ^{n}\left( n\geq 3\right) $ be a bounded
domain with sufficiently smooth boundary $\partial\Omega.$ We study the
problem (1.1)
\begin{equation*}
\left\{
\begin{array}{l}
-div\left(\left\vert\nabla u\right\vert^{p_{1}\left(
x\right)-2}\nabla u\right)-\sum\limits_{i=1}^{n}D_{i}\left(
\left\vert u\right\vert ^{p_{0}\left(
x\right)-2}D_{i}u\right)+c\left( x,u\right) =h\left( x\right),\hspace{0.3cm}x\in \Omega \\
u\mid _{\partial \Omega }=0%
\end{array}%
\right.
\end{equation*}
under the following conditions:

\begin{enumerate}
\item[\textit{\textbf{(U1)}}] $2\leq p_{0}^{-}\leq p_{0}\left( x\right) \leq
p_{0}^{+}<\infty ,$ $1< p_{1}^{-}\leq p_{1}\left( x\right) \leq
p_{1}^{+}<\infty $ \textit{and} $p_{0}\in C^{1}\left(
\bar{\Omega}\right) ,$ $p_{1}\in C^{0}\left( \bar{\Omega}\right).$

\item[\textit{\textbf{(U2)}}] \textit{There exists a measurable function} $%
\alpha :\Omega \longrightarrow \left[ 1,\infty \right) ,$ $1\leq \alpha
^{-}\leq \alpha \left( x\right) \leq \alpha ^{+}<\infty $ \textit{such that
the following inequalities hold}
\begin{equation*}
\left\vert c\left( x,\tau \right) \right\vert \leq c_{0}\left( x\right)
\left\vert \tau \right\vert ^{\alpha \left( x\right) -1}+c_{1}\left( x\right)
\end{equation*}
and
\begin{equation*}
c\left( x,\tau \right) \tau \geq c_{2}\left( x\right) \left\vert \tau
\right\vert ^{\alpha \left( x\right) },
\end{equation*}
a.e. $\left( x,\tau \right) \in \Omega \times
\mathbb{R}
.$

\textit{Here }$c\left( x,\tau \right) $ \textit{is a Caratheodory function and} $%
c_{i} $, $i=0,1,2$ \textit{are nonnegative measurable functions defined on} $%
\Omega $, \textit{besides for} $\varepsilon >0,$ $\alpha \left( x\right)
\geq p_{0}\left( x\right) +\varepsilon $ \textit{and} $c_{2}\left( x\right)
\geq \tilde{C}>0$, $x\in \Omega $ \textit{is satisfied.}
\end{enumerate}

In order to give the definition of weak solution of the problem
(1.1), we introduce the required spaces. For this, first we remind
some basic facts about generalized Lebesgue and Sobolev spaces [2,
8, 10, 12, 15].

Let $\Omega $ be a Lebesgue measurable subset of $%
\mathbb{R}
^{n}$ such that $\left\vert \Omega \right\vert >0$ (Throughout the paper, we
denote by $\left\vert \Omega \right\vert $ the Lebesgue measure of $\Omega $%
). By $M\left( \Omega \right) $ denote the family of all
measurable
functions $p:\Omega \longrightarrow \left[ 1,\infty \right] $ and by $%
M_{0}\left( \Omega \right) $,
\begin{equation*}
M_{0}\left( \Omega \right) :=\left\{ p\in M\left( \Omega \right)
:\ 1\leq p^{-}\leq p\left( x\right) \leq p^{+}<\infty ,\text{ a.e.
}x\in \Omega \right\} .
\end{equation*}
where $p^{-}:=\underset{\Omega }{%
ess}\inf \left\vert p\left( x\right) \right\vert,\text{ } p^{+}:=%
\underset{\Omega }{ess}\sup \left\vert p\left( x\right)
\right\vert$. \par For $p\in $ $M\left( \Omega \right) ,$ $\Omega
_{\infty }^{p}\equiv \Omega _{\infty }\equiv \left\{ x\in \Omega
|\text{ }p\left( x\right) =\infty \right\} $ then on the set of
all functions on $\Omega $ define the
functional $\sigma _{p}$ and $\left\Vert .\right\Vert _{p}$ by%
\begin{equation*}
\sigma _{p}\left( u\right) \equiv \int\limits_{\Omega \backslash \Omega
_{\infty }}\left\vert u\right\vert ^{p\left( x\right) }dx+\underset{\Omega
_{\infty }}{ess}\sup \left\vert u\left( x\right) \right\vert
\end{equation*}%
and%
\begin{equation*}
\left\Vert u\right\Vert _{L^{p\left( x\right) }\left( \Omega \right) }\equiv
\inf \left\{ \lambda >0:\text{ }\sigma _{p}\left( \frac{u}{\lambda }\right)
\leq 1\right\} .
\end{equation*}%
Clearly if $p\in L^{\infty }\left( \Omega \right) $ then $p\in
M_{0}\left( \Omega \right) $, $\sigma _{p}\left( u\right) \equiv
\int\limits_{\Omega }\left\vert u\right\vert ^{p\left( x\right)
}dx$ and the generalized
Lebesgue space is defined as follows:%
\begin{equation*}
L^{p\left( x\right) }\left( \Omega \right) := \left\{ u: u\text{ }
\text{is a measurable real-valued function such that}\text{ }
\sigma _{p}\left( u\right) <\infty \right\} .
\end{equation*}%
If $p^{-}>1,$ then the space $L^{p\left( x\right) }\left( \Omega
\right) $ becomes a reflexive and separable Banach space under the
norm $\left\Vert .\right\Vert _{L^{p\left( x\right) }\left( \Omega
\right) }$which is so-called Luxemburg norm. \par If $0<\left\vert
\Omega \right\vert <\infty ,$ and
$p_{1},$ $p_{2}\in M\left( \Omega \right) $ then the continuous embedding  $%
L^{p_{1}\left( x\right) }\left( \Omega \right) \subset L^{p_{2}\left(
x\right) }\left( \Omega \right) $ exists$\iff $ $p_{2}\left( x\right) \leq
p_{1}\left( x\right) $ for a.e $x\in \Omega .$

For $u\in L^{p\left( x\right) }\left( \Omega \right) $ and $v\in
L^{q\left( x\right) }\left( \Omega \right) $ where $p,$ $q\in $
$M_{0}\left( \Omega \right) $ and $\frac{1}{p\left( x\right)
}+\frac{1}{q\left( x\right) }=1$ the following inequalities holds

\begin{equation*}
\int\limits_{\Omega }\left\vert uv\right\vert dx\leq 2\left\Vert
u\right\Vert _{L^{p\left( x\right) }\left( \Omega \right) }\left\Vert
v\right\Vert _{L^{q\left( x\right) }\left( \Omega \right), } \text{
(generalized H\"{o}lder inequality)}  \tag{2.1}
\end{equation*}
and
\begin{equation*}
\min \{\left\Vert u\right\Vert _{L^{p\left( x\right) }\left( \Omega \right)
}^{p^{-}}, \left\Vert u\right\Vert _{L^{p\left( x\right) }\left( \Omega
\right) }^{p^{+}}\}\leq \sigma _{p}\left( u\right) \leq \max \{\left\Vert
u\right\Vert _{L^{p\left( x\right) }\left( \Omega \right) }^{p^{-}},
\left\Vert u\right\Vert _{L^{p\left( x\right) }\left( \Omega \right)
}^{p^{+}}\}  \tag{2.2}
\end{equation*}

Let $\Omega \subset
\mathbb{R}
^{n}$ be a bounded domain and $p\in L^{\infty }\left( \Omega \right) $ then
generalized Sobolev space is defined as follows:%
\begin{equation*}
W^{1,\text{ }p\left( x\right) }\left( \Omega \right) :=\left\{ u\in
L^{p\left( x\right) }\left( \Omega \right) :\text{ }\left\vert \nabla
u\right\vert \in L^{p\left( x\right) }\left( \Omega \right) \right\}
\end{equation*}%
and this space is a separable Banach space under the norm:%
\begin{equation*}
\left\Vert u\right\Vert _{W^{1,\text{ }p\left( x\right) }\left( \Omega
\right) }\equiv \left\Vert u\right\Vert _{L^{p\left( x\right) }\left( \Omega
\right) }+\left\Vert \nabla u\right\Vert _{L^{p\left( x\right) }\left(
\Omega \right) }
\end{equation*}%
$W_{0}^{1,\text{ }p\left( x\right) }\left( \Omega \right) $ defines as the
closure of $C_{0}^{\infty }\left( \Omega \right) $ in $W^{1,\text{ }p\left(
x\right) }\left( \Omega \right) .$ If $p^{-}>1$ then $W^{1,\text{ }p\left(
x\right) }\left( \Omega \right) $ and $W_{0}^{1,\text{ }p\left( x\right)
}\left( \Omega \right) $ are reflexive and separable Banach spaces. If $%
\partial \Omega $ is Lipschitz boundary$\ $and $p\in C^{0}\left( \bar{\Omega}%
\right) ,$ then equivalent norm in $W_{0}^{1,\text{ }p\left( x\right)
}\left( \Omega \right) $ is given by;%
\begin{equation*}
\left\Vert u\right\Vert _{W_{0}^{1,\text{ }p\left( x\right)
}\left( \Omega \right) }\equiv \left\Vert \nabla u\right\Vert
_{L^{p\left( x\right) }\left( \Omega \right)}\equiv
\sum_{i=1}^{n}\left\Vert D_{i}u\right\Vert_{L^{p\left( x\right)
}\left( \Omega \right)}.
\end{equation*}%
Let $p$, $q\in C\left( \bar{\Omega}\right) \cap M_{0}\left( \Omega
\right) $
and $p\left( x\right) <n,$ $q\left( x\right) <\frac{np\left( x\right) }{%
n-p\left( x\right) }\equiv p^{\ast }\left( x\right) $ is hold for all $x\in
\Omega ,$ then there is a continuous and compact embedding $W^{1,\text{ }%
p\left( x\right) }\left( \Omega \right) \hookrightarrow L^{q\left( x\right)
}\left( \Omega \right) .$

A function $p\in M_{0}\left( \Omega \right) $ is called
log-H\"{o}lder continuous if there is a constant $L$ such that the
inequality
\begin{equation*}
-\left\vert p\left( x\right) -p\left( y\right) \right\vert \log
\left\vert x-y\right\vert \leq L,\text{ \ \ \ \ }\forall
x,\text{}y\in \Omega
\end{equation*}%
\ holds. If $p$ is log-H\"{o}lder continuous and $q\in M_{0}\left(
\Omega \right) $ then we have the continuous embedding
$W^{1,\text{ }p\left( x\right) }\left( \Omega \right) \subset
L^{q\left( x\right) }\left( \Omega \right) $ for all $q\leq
p^{\ast }.$ \newline For more details and embedding results for
these spaces see [2, 8, 10-12, 15]. \vspace{0.2cm}

We now define some function classes which are required to study the problem
(1.1). These classes are nonlinear spaces which are the generalization of
the nonlinear spaces with constant exponent studied in [21-24](see also
references of them). We also note that the necessary properties of these
spaces are presented in Section 3.

\begin{definition}
Let $\Omega \subset
\mathbb{R}
^{n}\left( n\geq 2\right) $ be a bounded domain with Lipschitz
boundary and $\gamma ,$ $\beta $ $\in P_{0}\left( \Omega \right)
.$ We introduce $S_{1,\gamma \left( x\right) ,\beta \left(
x\right) }\left( \Omega \right) ,$ the class of functions
$u:\Omega \rightarrow
\mathbb{R}
$ and the functional $[.]_{S_{\gamma ,\beta }}:S_{1,\gamma \left( x\right)
,\beta \left( x\right) }\left( \Omega \right) \longrightarrow
\mathbb{R}
_{+}$ as follows:%
\begin{equation*}
S_{1,\gamma \left( x\right) ,\beta \left( x\right) }\left( \Omega \right)
:=\left\{ u\in L^{1}\left( \Omega \right) :\int\limits_{\Omega }\left\vert
u\right\vert ^{\gamma \left( x\right) +\beta \left( x\right)
}dx+\sum_{i=1}^{n}\int\limits_{\Omega }\left\vert u\right\vert ^{\gamma
\left( x\right) }\left\vert D_{i}u\right\vert ^{\beta \left( x\right)
}dx<\infty \right\} ,
\end{equation*}%
\begin{equation*}
\lbrack u]_{S_{\gamma ,\beta }}:=\inf \left\{ \lambda >0:\int\limits_{\Omega
}\left\vert \frac{u}{\lambda }\right\vert ^{\gamma \left( x\right) +\beta
\left( x\right) }dx+\sum_{i=1}^{n}\left( \int\limits_{\Omega }\left\vert
\frac{\left\vert u\right\vert ^{\frac{\gamma \left( x\right) }{\beta \left(
x\right) }}D_{i}u}{\lambda ^{\frac{\gamma \left( x\right) }{\beta \left(
x\right) }+1}}\right\vert ^{\beta \left( x\right) }\right) dx\leq 1\right\} .
\end{equation*}
\end{definition}

$[.]_{S_{\gamma ,\beta }}$ defines a pseudo-norm on $S_{1,\gamma
\left( x\right) ,\beta \left( x\right) }\left( \Omega \right) ,$
actually it can be clearly seen that $[.]_{S_{\gamma ,\beta }}$
fulfills all conditions of
pseudo-norm (pn) see [21] i.e. $[u]_{S_{\gamma ,\beta }}\geq 0,$ $%
u=0\Rightarrow \lbrack u]_{S_{\gamma ,\beta }}=0,$ $[u]_{S_{\gamma ,\beta
}}\neq \lbrack v]_{S_{\gamma ,\beta }}\Rightarrow u\neq v$ and $%
[u]_{S_{\gamma ,\beta }}=0\Rightarrow u=0.$

\medskip

Let $S_{1,\gamma \left( x\right) ,\beta \left( x\right) }\left(
\Omega \right) $ be the space given in the Definition 2.1 and
$\theta \left( x\right)\in M_{0}\left( \Omega \right)$, we denote
$S_{1,\gamma \left( x\right) ,\beta \left( x\right) ,\theta \left(
x\right) }\left( \Omega \right) ,$ the class of functions
$u:\Omega \rightarrow \mathbb{R} $ by the following intersection
\begin{equation*}
S_{1,\gamma \left( x\right) ,\beta \left( x\right) ,\theta \left( x\right)
}\left( \Omega \right) :=S_{1,\gamma \left( x\right) ,\beta \left( x\right)
}\left( \Omega \right) \cap L^{\theta \left( x\right) }\left( \Omega \right)
,  \tag{2.3}
\end{equation*}%
with the pseudo-norm
\begin{equation*}
\lbrack u]_{S_{\gamma ,\beta ,\theta} }:=[u]_{S_{\gamma ,\beta }}+\left\Vert
u\right\Vert _{L^{\theta \left( x\right) }\left( \Omega \right) },\text{ \ }%
\forall u\in S_{1,\gamma \left( x\right) ,\beta \left( x\right) ,\theta
\left( x\right) }\left( \Omega \right) .
\end{equation*}

We now state a proposition which can be easily proved by the embedding
results for the Lebesgue spaces with variable exponent.

\begin{proposition}
If $\gamma $, $\beta $, $\theta $ $\in M_{0}\left( \Omega \right) $ and $%
\theta \left( x\right) \geq \gamma \left( x\right) +\beta \left( x\right)
+\varepsilon _{0}$ a.e. $x\in \Omega $ for some $\varepsilon _{0}>0,$ then
we have the following equivalence;
\begin{equation*}
S_{1,\gamma \left( x\right) ,\beta \left( x\right) ,\theta \left( x\right)
}\left( \Omega \right) \equiv \left\{ u\in L^{1}\left( \Omega \right) :\Re
^{\gamma ,\beta ,\theta }\left( u\right) :=\int\limits_{\Omega }\left\vert
u\right\vert ^{\theta \left( x\right) }dx+\sum_{i=1}^{n}\int\limits_{\Omega
}\left\vert u\right\vert ^{\gamma \left( x\right) }\left\vert
D_{i}u\right\vert ^{\beta \left( x\right) }dx<\infty \right\} ,
\end{equation*}%
with the pseudo-norm
\begin{equation*}
\lbrack u]_{S_{\gamma ,\beta ,\theta }}\equiv \inf \left\{ \lambda
>0:\int\limits_{\Omega }\left\vert \frac{u}{\lambda }\right\vert ^{\theta
\left( x\right) }dx+\sum_{i=1}^{n}\left( \int\limits_{\Omega }\left\vert
\frac{\left\vert u\right\vert ^{\frac{\gamma \left( x\right) }{\beta \left(
x\right) }}D_{i}u}{\lambda ^{\frac{\gamma \left( x\right) }{\beta \left(
x\right) }+1}}\right\vert ^{\beta \left( x\right) }\right) dx\leq 1\right\} .
\end{equation*}
\end{proposition}

\medskip\noindent Also denote the dual spaces, $W^{-1,\text{ }q_{0}\left( x\right)
}\left( \Omega \right) :=\left( W_{0}^{1,\text{ }p_{0}\left( x\right)
}\left( \Omega \right) \right) ^{\ast },$ $W^{-1,\text{ }q_{1}\left(
x\right) }\left( \Omega \right) :=\left( W_{0}^{1,\text{ }p_{1}\left(
x\right) }\left( \Omega \right) \right) ^{\ast }$ and $L^{\alpha ^{\prime
}\left( x\right) }\left( \Omega \right) :=\left( L^{\alpha \left( x\right)
}\left( \Omega \right) \right) ^{\ast },$ here $q_{0}\left( x\right) :=\frac{%
p_{0}\left( x\right) }{p_{0}\left( x\right) -1}, $ $q_{1}\left( x\right) :=%
\frac{p_{1}\left( x\right) }{p_{1}\left( x\right) -1}$ and $\alpha ^{\prime
}\left( x\right) :=\frac{\alpha \left( x\right) }{\alpha \left( x\right) -1}%
. $

We investigate the problem (1.1) for functions $h\in W^{-1,\text{ }%
q_{0}\left( x\right) }\left( \Omega \right) +L^{\alpha ^{\prime }\left(
x\right) }\left( \Omega \right) +W^{-1,\text{ }q_{1}\left( x\right) }\left(
\Omega \right) .$ Let us denote $Q\left( \Omega \right)$ by
\begin{equation*}
Q\left( \Omega \right) :=\mathring{S}_{1,q_{0}\left( x\right) \left(
p_{0}\left( x\right) -2\right) ,q_{0}\left( x\right) ,\alpha \left( x\right)
}\left( \Omega \right) \cap W_{0}^{1,\text{ }p_{1}\left( x\right) }\left(
\Omega \right)
\end{equation*}%
where $\mathring{S}_{1,q_{0}\left( x\right) \left( p_{0}\left(
x\right) -2\right) ,q_{0}\left( x\right) ,\alpha \left( x\right)
}\left( \Omega \right):=\left\{ u\in S_{1,q_{0}\left( x\right)
\left( p_{0}\left( x\right) -2\right) ,q_{0}\left( x\right)
,\alpha \left( x\right) }\left( \Omega \right) : u\mid _{\partial
\Omega }=0\right\} .$ Noticed that the condition on $\alpha \left(
x\right)$ in \emph{(U2)} indicates that that Proposition 2.2 is
valid for $\mathring{S}_{1,q_{0}\left( x\right) \left( p_{0}\left(
x\right) -2\right) ,q_{0}\left( x\right) ,\alpha \left( x\right)
}\left( \Omega \right)$.

\medskip\noindent We are ready to give the definition of the weak
solution of problem (1.1). \medskip

\begin{definition}
If the function $u\in Q\left( \Omega \right) $ satisfies the following
equality%
\begin{align}
&\int\limits_{\Omega }\left(\left\vert\nabla
u\right\vert^{p_{1}\left( x\right)-2}\nabla u\right)\cdot\nabla
vdx+\sum_{i=1}^{n}\int\limits_{\Omega }\left( \left\vert
u\right\vert ^{p_{0}\left( x\right) -2}D_{i}u\right)
D_{i}vdx+\notag{}\\
&+\int\limits_{\Omega }c\left( x,u\right) vdx=\int\limits_{\Omega
}hvdx  \tag{2.4}
\end{align}%
for every $v\in W_{0}^{1,\text{ }p_{0}\left( x\right) }\left( \Omega \right)
\cap W_{0}^{1,\text{ }p_{1}\left( x\right) }\left( \Omega \right) \cap
L^{\alpha \left( x\right) }\left( \Omega \right) ,$ then the function $u$ is
called the weak solution of the problem (1.1).
\end{definition}

\noindent Note that it is clear under the conditions \emph{(U1)}
and \emph{(U2)}, all the integrals in (2.4) make sense.

\noindent Now we state the main theorem of this article that is
the solvability theorem for problem (1.1): \medskip

\begin{theorem}
\textbf{(Existence Theorem)} Let the conditions
\emph{(U1)}-\emph{(U2)}
fulfill and $c_{0}, $ $c_{2}\in L^{\infty }\left( \Omega \right) ,$ $%
c_{1}\in $ $L^{\alpha ^{\prime }\left( x\right) }\left( \Omega \right) .$
Then for every $h\in $ $W^{-1,\text{ }q_{0}\left( x\right) }\left( \Omega
\right) +L^{\alpha ^{\prime }\left( x\right) }\left( \Omega \right) +W^{-1,%
\text{ }q_{1}\left( x\right) }\left( \Omega \right) ,$ problem (1.1) has a
weak solution in the space $Q\left( \Omega \right) .$
\end{theorem}

We will use the following solvability theorem [24] to prove the Theorem 2.4.
Let $X$ and $Y_{0}$ be reflexive Banach spaces, $Y$ is an arbitrary Banach
space and $S_{gY_{0}}$ is pn-space(pseudo-norm space) [21]. Let $%
A:X\longrightarrow X^{\ast }$ and $B:S_{gY_{0}}\longrightarrow Y$ be
nonlinear operators.

\medskip Assume that the following conditions are satisfied:

\begin{enumerate}
\item[1)] $A:X\longrightarrow X^{\ast }$ is a pseudo-monotone operator, i.e.,

\begin{enumerate}
\item[(i)] $A $ is bounded operator and

\item[(ii)] the conditions $u_{m}\overset{X}{\rightharpoonup }u_{0}$ and
limsup$\left\langle A\left( u_{m}\right) ,\text{ }u_{m}-u\right\rangle \leq
0 $ imply
\begin{equation*}
\text{liminf}\left\langle A\left( u_{m}\right) ,\text{ }u_{m}-v\right\rangle
\geq \left\langle A\left( u\right) ,\text{ }u-v\right\rangle ,\text{ \ }%
\forall v\in X.
\end{equation*}
\end{enumerate}

\item[2)] $B:S_{gY_{0}}\longrightarrow Y$ is weakly compact. Furthermore,
there exists a mapping $B_{0}:X_{0}\cap S_{gY_{0}}\longrightarrow
Y_{2}\subseteq Y$ such that $B_{0}$ is weakly compact from $X_{0}\cap
S_{gY_{0}}$ to $Y_{2}$ where $X_{0}$ is a separable topological vector space
which is dense in $X,$ $Y^{\ast }$ and $S_{gY_{0}}$ and there exists a
continuous nondecreasing function $\varphi :%
\mathbb{R}
_{+}^{1}\longrightarrow
\mathbb{R}
_{+}^{1},$ such that $\varphi \in C^{0}$ and
\begin{equation*}
\left\langle B\left( u\right) ,\text{ }u\right\rangle =\varphi \left(
\left\Vert B_{0}\left( u\right) \right\Vert _{Y_{2}}\right) ,\text{ \ }%
\forall u\in S_{gY_{0}}.
\end{equation*}

\item[3)] The operator $T:=A+B$ is coercive in the generalized sense on $%
X_{0},$ i.e. for each $u\in X_{0}$ with $\left\Vert u\right\Vert _{X}$, $%
[u]_{S_{gY_{0}}}\geq M$ we have
\begin{equation*}
\left\langle T\left( u\right) ,\text{ }u\right\rangle =\left\langle A\left(
u\right) +B\left( u\right) ,\text{ }u\right\rangle \geq \lambda _{0}\left(
\left\Vert u\right\Vert _{X}\right) \left\Vert u\right\Vert _{X}+\lambda
_{1}\left( [u]_{S_{gY_{0}}}\right) [u]_{S_{gY_{0}}}
\end{equation*}%
where $\lambda _{0},$ $\lambda _{1}\in C^{0}$ such that as $\tau \nearrow
\infty ,$ $\lambda _{0}\left( \tau \right) ,$ $\lambda _{1}\left( \tau
\right) \nearrow \infty $ and $M>0$ is some number.
\end{enumerate}

\begin{theorem}
$[24]$ Let conditions 1)-3) be satisfied. Then the equation%
\begin{equation*}
T\left( u\right) =A\left( u\right) +B\left( u\right) =y,\text{ \ }y\in
X^{\ast }+Y,
\end{equation*}%
is solvable in $X\cap S_{gY_{0}}$ for any $y\in X^{\ast }+Y$ satisfying%
\begin{equation*}
\sup \left\{ \frac{\left\langle y,\text{ }u\right\rangle }{\left\Vert
u\right\Vert _{X}+[u]_{S_{gY_{0}}}}:\text{ }u\in X_{0}\right\} <\infty .
\end{equation*}
\end{theorem}

\section{Preliminary Results}

In this section, we study the function classes which are defined
in Section 2 that actually is required to investigate the problem
(1.1). First, we establish some integral inequalities to realize
the structure of these spaces. Afterwards, we show that these
spaces are complete metric spaces. Moreover, we prove some lemmas
and theorems on continuous and compact embedding etc. for these
spaces and also indicate their relation with the Lebesgue and
Sobolev spaces with variable exponent.

\subsection{Some Integral Inequalities}

In this subsection, we derive some inequalities which are given as lemmas.
As the proofs of these lemmas can be obtained easily by using Young's, H\"{o}%
lder inequalities and by calculations (and also see [20, Lemma 4.1]) hence,
we skip the proofs for the sake of brevity. Throughout this section we
assume that $\Omega \subset
\mathbb{R}
^{n}\left( n\geq 2\right) $ is a bounded domain with Lipschitz boundary.

\begin{lemma}
Let $\zeta $, $\xi$ $\in M_{0}\left( \Omega \right)$ and $\zeta
\left(
x\right) \geq \xi \left( x\right) $ a.e. $x\in \Omega .$ Then the inequality%
\begin{equation}
\int\limits_{\Omega }\left\vert u\right\vert ^{\xi \left( x\right) }dx\leq
\int\limits_{\Omega }\left\vert u\right\vert ^{\zeta \left( x\right)
}dx+\left\vert \Omega \right\vert ,\text{ \ \ }\forall u\in L^{\zeta \left(
x\right) }\left( \Omega \right)  \tag{3.1}
\end{equation}%
holds.
\end{lemma}

\begin{lemma}
Assume that $\zeta $ $\in M_{0}\left( \Omega \right)$ and the
numbers $\eta$ and $\epsilon$ satisfy $\eta\geq 1,$ $\epsilon >0.$
Then for every $u\in
L^{\zeta \left( x\right) +\epsilon }\left( \Omega \right) $%
\begin{equation}
\int\limits_{\Omega }\left\vert u\right\vert ^{\zeta \left(
x\right) }\left\vert \ln \left\vert u\right\vert \right\vert
^{\eta }dx\leq N_{1}\int\limits_{\Omega }\left\vert u\right\vert
^{\zeta \left( x\right) +\epsilon }dx+N_{2}  \tag{3.2}
\end{equation}%
is satisfied. Here $N_{1}\equiv N_{1}\left( \epsilon ,\eta \right) >0$ and $%
N_{2}\equiv N_{2}\left( \epsilon ,\eta ,\left\vert \Omega
\right\vert \right) >0$ are constants.
\end{lemma}

\begin{corollary}
Let $\zeta $, $\eta $ $\in M_{0}\left( \Omega \right)$. Then for
$\epsilon
>0,$ the inequality%
\begin{equation}
\int\limits_{\Omega }\left\vert u\right\vert ^{\zeta \left(
x\right) }\left\vert \ln \left\vert u\right\vert \right\vert
^{\eta \left( x\right) }dx\leq N_{3}\int\limits_{\Omega
}\left\vert u\right\vert ^{\zeta \left( x\right) +\epsilon
}dx+N_{4},\text{ \ \ }\forall u\in L^{\zeta \left( x\right)
+\epsilon }\left( \Omega \right)  \tag{3.3}
\end{equation}%
holds. Here $N_{3}\equiv N_{3}\left( \epsilon ,\text{ }\eta
^{+}\right) >0$ and $N_{4}\equiv N_{4}\left( \epsilon , \eta ^{+},
\left\vert \Omega \right\vert \right) >0$ are constants.
\end{corollary}

\subsection{Generalized Nonlinear Spaces and Embedding
Theorems}

In this section, we examine the properties of the spaces $S_{1,\gamma \left(
x\right) ,\beta \left( x\right) ,\theta \left( x\right) }\left( \Omega
\right) $ and their connection with the known spaces. Investigating most of
boundary value problems on its own space leads to obtain better results.
Henceforth considered problem (1.1) is investigated on its own space (i.e. $%
S_{1,\gamma \left( x\right) ,\beta \left( x\right) ,\theta \left( x\right)
}\left( \Omega \right) $). Unlike linear boundary value problems, the sets
generated by nonlinear problems are subsets of linear spaces, but not
possessing the linear structure [21-24].\newline
Note that, from now on unless additional conditions are imposed, all the
functions $\gamma ,$ $\beta $ and $\theta $ will satisfy the conditions
given in Proposition 2.2.

\begin{lemma}
Let $u\in $ $S_{1,\gamma \left( x\right) ,\beta \left( x\right) ,\theta
\left( x\right) }\left( \Omega \right) $ and $\lambda _{u}:=[u]_{S_{\gamma
,\beta ,\theta }}$ then the following inequality%
\begin{equation*}
\max \{\lambda _{u}^{\gamma ^{-}+\beta ^{-}},\lambda _{u}^{\theta
^{+}}\}\geq \Re ^{\gamma ,\beta ,\theta }\left( u\right) \geq \min \{\lambda
_{u}^{\gamma ^{-}+\beta ^{-}},\lambda _{u}^{\theta ^{+}}\}
\end{equation*}%
holds.
\end{lemma}

\begin{proof}
For $u\in $ $S_{1,\gamma \left( x\right) ,\beta \left( x\right) ,\theta
\left( x\right) }\left( \Omega \right) ,$
\begin{equation*}
\Re ^{\gamma ,\beta ,\theta }\left( u\right) =\int\limits_{\Omega
}\left\vert u\right\vert ^{\theta \left( x\right)
}dx+\sum_{i=1}^{n}\int\limits_{\Omega }\left\vert u\right\vert ^{\gamma
\left( x\right) }\left\vert D_{i}u\right\vert ^{\beta \left( x\right) }dx
\end{equation*}%
\begin{equation*}
=\int\limits_{\Omega }\lambda _{u}^{\theta \left( x\right) }\left\vert \frac{%
u}{\lambda _{u}}\right\vert ^{\theta \left( x\right)
}dx+\sum_{i=1}^{n}\int\limits_{\Omega }\lambda _{u}^{\gamma \left( x\right)
+\beta \left( x\right) }\left\vert \frac{\left\vert u\right\vert ^{\frac{%
\gamma \left( x\right) }{\beta \left( x\right) }}D_{i}u}{\lambda _{u}^{\frac{%
\gamma \left( x\right) }{\beta \left( x\right) }+1}}\right\vert ^{\beta
\left( x\right) }dx,
\end{equation*}%
if $\lambda _{u}\geq 1,$ we have
\begin{equation*}
\geq \lambda _{u}^{\theta ^{-}}\int\limits_{\Omega }\left\vert \frac{u}{%
\lambda _{u}}\right\vert ^{\theta \left( x\right) }dx+\lambda _{u}^{\gamma
^{-}+\beta ^{-}}\sum_{i=1}^{n}\int\limits_{\Omega }\left\vert \frac{%
\left\vert u\right\vert ^{\frac{\gamma \left( x\right) }{\beta \left(
x\right) }}D_{i}u}{\lambda _{u}^{\frac{\gamma \left( x\right) }{\beta \left(
x\right) }+1}}\right\vert ^{\beta \left( x\right) }dx
\end{equation*}%
\begin{equation*}
\geq \lambda _{u}^{\gamma ^{-}+\beta ^{-}}\left( \int\limits_{\Omega
}\left\vert \frac{u}{\lambda _{u}}\right\vert ^{\theta \left( x\right)
}dx+\sum_{i=1}^{n}\int\limits_{\Omega }\left\vert \frac{\left\vert
u\right\vert ^{\frac{\gamma \left( x\right) }{\beta \left( x\right) }}D_{i}u%
}{\lambda _{u}^{\frac{\gamma \left( x\right) }{\beta \left( x\right) }+1}}%
\right\vert ^{\beta \left( x\right) }dx\right),
\end{equation*}%
from the definition of $[.]_{S_{\gamma ,\beta },\theta }$ and last
inequality, we obtain%
\begin{equation*}
\Re ^{\gamma ,\beta ,\theta }\left( u\right)\geq \lambda _{u}^{\gamma
^{-}+\beta ^{-}}.
\end{equation*}%
Obviously, if $0<\lambda _{u}<1$ then $\Re ^{\gamma ,\beta ,\theta }\left(
u\right) \geq \lambda _{u}^{\theta ^{+}}.$ Similarly, we can show the other
side of the inequality thus, the proof is complete.
\end{proof}

\begin{theorem}
Assume that $p \in M_{0}\left( \Omega \right)$ and $p\left(
x\right) \geq
\theta \left( x\right) $ a.e. $x\in \Omega .$ Then, we have the embedding%
\begin{equation}
W^{1,\text{ }p\left( x\right) }\left( \Omega \right) \subset S_{1,\gamma
\left( x\right) ,\beta \left( x\right) ,\theta \left( x\right) }\left(
\Omega \right) .  \tag{3.4}
\end{equation}
\end{theorem}

\begin{proof}
Let $u\in W^{1,\text{ }p\left( x\right) }\left( \Omega \right) ,$
as a consequence of Lemma 3.4 to obtain the embedding (3.4) it is
sufficient to show that $\Re ^{\gamma ,\beta ,\theta }\left(
u\right) $ is finite (i.e. $\Re ^{\gamma ,\beta ,\theta }\left(
u\right) <\infty$)
\begin{equation*}
\Re ^{\gamma ,\beta ,\theta }\left( u\right) =\int\limits_{\Omega
}\left\vert u\right\vert ^{\theta \left( x\right)
}dx+\sum_{i=1}^{n}\int\limits_{\Omega }\left\vert u\right\vert ^{\gamma
\left( x\right) }\left\vert D_{i}u\right\vert ^{\beta \left( x\right) }dx,
\end{equation*}%
by Lemma 3.1 and using Young's inequality we get%
\begin{equation*}
\leq \int\limits_{\Omega }\left\vert u\right\vert ^{p\left( x\right)
}dx+\left\vert \Omega \right\vert +\sum_{i=1}^{n}\int\limits_{\Omega
}\left\vert D_{i}u\right\vert ^{p\left( x\right) }dx+n\int\limits_{\Omega
}\left\vert u\right\vert ^{\frac{p\left( x\right) \gamma \left( x\right) }{%
p\left( x\right) -\beta \left( x\right) }}dx,
\end{equation*}%
estimating the third integral on the right side of the last inequality by
using Lemma 3.1, we obtain%
\begin{equation*}
\Re ^{\gamma ,\beta ,\theta }\left( u\right) \leq \left( n+1\right) \left(
\sigma _{p}\left( u\right) +\left\vert \Omega \right\vert \right)
+\sum_{i=1}^{n}\sigma _{p}\left( D_{i}u\right) ,
\end{equation*}%
thus, we get the desired result from the last inequality and (2.2).

If $p\left( x\right) =\theta \left( x\right) $ a.e. $x\in \Omega
,$ by using the same operations as done above we can obtain (3.4).
\end{proof}

\medskip We omit the proof of the following lemma as it is straightforward.

\begin{lemma}
Let $\gamma ,$ $\beta :\Omega \longrightarrow \left[ 1,\text{ }\infty
\right) $ be functions satisfying $1\leq \gamma ^{-}\leq \gamma \left(
x\right) \leq \gamma ^{+}<\infty $, $1\leq \beta ^{-}\leq \beta \left(
x\right) \leq \beta ^{+}<\infty $ a.e. $x\in \Omega $ and $\gamma ,$ $\beta
\in C^{1}\left( \bar{\Omega}\right) $. Then the function $\varphi :\Omega
\times
\mathbb{R}
\longrightarrow
\mathbb{R}
,$ $\varphi \left( x,t\right) :=\left\vert t\right\vert ^{\frac{\gamma
\left( x\right) }{\beta \left( x\right) }}t$ satisfies the following:

\begin{enumerate}
\item[(i)] For every fixed $x_{0}\in \Omega ,$ $\varphi \left(
x_{0},.\right) :%
\mathbb{R}
\longrightarrow
\mathbb{R}
$ is continuously differentiable; has an inverse and inverse function is
also continuously differentiable.

\item[(ii)] For every fixed $t_{0}\in
\mathbb{R}
-\left\{ 0\right\} ,$ $\varphi $ and $\varphi ^{-1}$ is continuous on $%
\Omega ,$ and for $\forall i=\overline{1,n},$ partial derivatives $\varphi
_{x_{i}}\left( x,t_{0}\right) ,$ $\varphi _{x_{i}}^{-1}\left( x,t_{0}\right)
$ exist and are continuous.
\end{enumerate}
\end{lemma}

\begin{definition}
Let $\eta $ $\in M_{0}\left( \Omega \right),$ we introduce $L^{1,\text{ }%
\eta \left( x\right) }\left( \Omega \right)$, the the class of functions $%
u:\Omega \rightarrow \mathbb{R} $
\begin{equation*}
L^{1,\text{ }\eta \left( x\right) }\left( \Omega \right) \equiv \left\{ u\in
L^{1}\left( \Omega \right) |\text{ }D_{i}u\in L^{\eta \left( x\right)
}\left( \Omega \right) ,\text{ }i=\overline{1,n}\right\}.\footnote{%
Note that this space is not Banach differently from the space $W^{1,\text{ }%
\eta \left( x\right) }\left( \Omega \right) [8]. $}
\end{equation*}
\end{definition}

\begin{theorem}
Let the functions $\gamma ,$ $\beta $ and $\varphi $ satisfy the
conditions of Lemma 3.6 and $L^{1,\text{ }\beta \left( x\right)
}\left( \Omega \right) $ be the space given in Definition 3.7.
Then, $\varphi $ is a bijective mapping between $S_{1,\gamma
\left( x\right) ,\beta \left(
x\right) ,\theta \left( x\right) }\left( \Omega \right) $ and $L^{1,\text{ }%
\beta \left( x\right) }\left( \Omega \right) \cap L^{\psi \left( x\right)
}\left( \Omega \right) $ where $\psi \left( x\right) :=\frac{\theta \left(
x\right) \beta \left( x\right) }{\gamma \left( x\right) +\beta \left(
x\right) }.$
\end{theorem}

\begin{proof}
First let us verify that for every $u\in S_{1,\gamma \left( x\right) ,\beta
\left( x\right) ,\theta \left( x\right) }\left( \Omega \right) ,$%
\begin{equation*}
v:=\left\vert u\right\vert ^{\frac{\gamma \left( x\right) }{\beta
\left( x\right) }}u=\varphi \left( u\right) \in L^{1,\text{ }\beta
\left( x\right) }\left( \Omega \right) \cap L^{\psi \left(
x\right) }\left( \Omega \right)\footnote{From now on, for
simplicity we denote $\varphi \left(x, u\right):=\varphi \left(
u\right)=\left\vert u\right\vert ^{\frac{\gamma \left( x\right)
}{\beta \left( x\right) }}u$.}
\end{equation*}
to show this, from the definition of the spaces $L^{1,\text{
}\beta \left( x\right) }\left( \Omega \right) $ and $L^{\psi
\left( x\right) }\left( \Omega \right) ,$ it is sufficient to
prove that $\forall i=\overline{1,n},$ $\sigma _{\beta }\left(
D_{i}v\right) $ and $\sigma _{\psi }\left( v\right) $ are finite.
As
\begin{equation*}
\sigma _{\psi }\left( v\right) =\int\limits_{\Omega }\left\vert v\right\vert
^{\psi \left( x\right) }dx=\int\limits_{\Omega }\left\vert u\right\vert ^{%
\frac{\psi \left( x\right) \left( \gamma \left( x\right) +\beta \left(
x\right) \right) }{\beta \left( x\right) }}dx=\int\limits_{\Omega
}\left\vert u\right\vert ^{\theta \left( x\right) }dx,
\end{equation*}%
the above equation ensures that $\sigma _{\psi }\left( v\right) $ is finite.

\noindent Now for $\forall i=\overline{1,n},$ let us show that
$\sigma _{\beta }\left( D_{i}v\right) $ is finite,
\begin{align*}
&\sigma _{\beta }\left( D_{i}v\right) =\int\limits_{\Omega
}\left\vert D_{i}v\right\vert ^{\beta \left( x\right)
}dx=\int\limits_{\Omega }\left\vert D_{i}\left( \left\vert
u\right\vert ^{\frac{\gamma \left( x\right) }{\beta \left(
x\right) }}u\right) \right\vert ^{\beta \left( x\right) }dx\\
&=\int\limits_{\Omega }\left\vert \left( \tfrac{\gamma \left(
x\right) +\beta
\left( x\right) }{\beta \left( x\right) }\right) \left\vert u\right\vert ^{%
\frac{\gamma \left( x\right) }{\beta \left( x\right)
}}D_{i}u+D_{i}\left( \tfrac{\gamma \left( x\right) }{\beta \left(
x\right) }\right) \left\vert u\right\vert ^{\frac{\gamma \left(
x\right) }{\beta \left( x\right) }}u\ln \left\vert u\right\vert
\right\vert ^{\beta \left( x\right) }dx
\end{align*}%
estimating the corresponding coefficients and right hand side of above
equation by using Corollary 3.3, we obtain%
\begin{equation*}
\leq C_{1}\int\limits_{\Omega }\left\vert u\right\vert ^{\gamma \left(
x\right) }\left\vert D_{i}u\right\vert ^{\beta \left( x\right)
}dx+C_{2}\int\limits_{\Omega }\left\vert u\right\vert ^{\theta \left(
x\right) }dx+C_{3},
\end{equation*}%
here $C_{1}=C_{1}\left( \beta ^{\pm },\gamma ^{+}\right) ,$ $%
C_{2}=C_{2}\left( \beta ^{\pm },\gamma ^{+},\text{ }\left\Vert \gamma
\right\Vert _{C^{1}\left( \bar{\Omega}\right) },\left\Vert \beta \right\Vert
_{C^{1}\left( \bar{\Omega}\right) },\varepsilon _{0}\right) $ and $%
C_{3}=C_{3}\left( \beta ^{+},\left\vert \Omega \right\vert ,\varepsilon
_{0}\right) >0$ are constants. ($\varepsilon _{0}>0,$ comes from the
Proposition 2.2 which satisfy $\theta \left( x\right) \geq \gamma \left(
x\right) +\beta \left( x\right) +\varepsilon _{0}$).

For $C_{4}:=\max \left\{ C_{1},C_{2}\right\} ,$ we have%
\begin{align*}
\sigma _{\beta }\left( D_{i}v\right) &\leq C_{4}\left(
\int\limits_{\Omega }\left\vert u\right\vert ^{\theta \left(
x\right) }dx+\sum_{i=1}^{n}\int\limits_{\Omega }\left\vert
u\right\vert ^{\gamma \left( x\right) }\left\vert
D_{i}u\right\vert ^{\beta \left( x\right) }dx\right) +C_{3}\\
&=C_{4}\Re ^{\gamma ,\beta ,\theta }\left( u\right) +C_{3}
\tag{3.5}
\end{align*}%
since $u\in S_{1,\gamma \left( x\right) ,\beta \left( x\right)
,\theta \left( x\right) }\left( \Omega \right) ,$ so we get the
desired result by (3.5).

\noindent Now, conversely we need to show that for $\forall v\in L^{1,\text{
}\beta \left( x\right) }\left( \Omega \right) \cap L^{\psi \left( x\right)
}\left( \Omega \right) $%
\begin{equation*}
w:=\left\vert v\right\vert ^{-\frac{\gamma \left( x\right)
}{\gamma \left( x\right) +\beta \left( x\right) }}v=\varphi
^{-1}\left( v\right) \in S_{1,\gamma \left( x\right) ,\beta \left(
x\right) ,\theta \left( x\right) }\left( \Omega \right).
\end{equation*}%
From the definition of the space $S_{1,\gamma \left( x\right)
,\beta \left( x\right) ,\theta \left( x\right) }\left( \Omega
\right) ,$ it is sufficient to prove $\Re ^{\gamma ,\beta ,\theta
}\left( w\right) $ is finite. By using similar process and results
as mentioned above, we obtain
\begin{align*}
&\Re ^{\gamma ,\beta ,\theta }\left( w\right) =\int\limits_{\Omega
}\left\vert w\right\vert ^{\theta \left( x\right)
}dx+\sum_{i=1}^{n}\int\limits_{\Omega }\left\vert w\right\vert
^{\gamma \left( x\right) }\left\vert D_{i}w\right\vert ^{\beta
\left( x\right) }dx\\
&=\int\limits_{\Omega }\left\vert v\right\vert ^{\frac{\beta
\left( x\right)
\theta \left( x\right) }{\gamma \left( x\right) +\beta \left( x\right) }%
}dx+\sum_{i=1}^{n}\int\limits_{\Omega }\left\vert v\right\vert ^{\frac{%
\gamma \left( x\right) \beta \left( x\right) }{\gamma \left( x\right) +\beta
\left( x\right) }}\left\vert D_{i}\left( \left\vert v\right\vert ^{-\frac{%
\gamma \left( x\right) }{\gamma \left( x\right) +\beta \left( x\right) }%
}v\right) \right\vert ^{\beta \left( x\right) }dx\\
&\leq C_{5}\sum_{i=1}^{n}\int\limits_{\Omega }\left\vert
D_{i}v\right\vert ^{\beta \left( x\right)
}dx+C_{6}\int\limits_{\Omega }\left\vert v\right\vert ^{\psi
\left( x\right) }dx+C_{7},  \tag{3.6}
\end{align*}%
here $C_{5}=C_{5}\left( \beta ^{+}\right) >0,$ $C_{6}=C_{6}\left( \beta
^{+},\varepsilon _{1},\left\Vert \gamma \right\Vert _{C^{1}\left( \bar{\Omega%
}\right) },\left\Vert \beta \right\Vert _{C^{1}\left(
\bar{\Omega}\right) }\right) >0$ and $C_{7}=C_{7}\left( \beta
^{+},\varepsilon _{1},\left\vert \Omega \right\vert \right) >0$
are constants. \newline Since $v\in L^{1,\text{ }\beta \left(
x\right) }\left( \Omega \right) \cap L^{\psi \left( x\right)
}\left( \Omega \right) ,$ thus $w\in S_{1,\gamma \left( x\right)
,\beta \left( x\right) ,\theta \left( x\right) }\left( \Omega
\right)$ by (3.6).

To end the proof, it now remains to verify that $\varphi$ is
bijective, as we have shown in Lemma 3.6 that for fixed $x_{0}\in
\Omega $, $\varphi \left( t\right) :=\varphi \left(
x_{0},t\right)$ and $\varphi ^{-1}\left(
\tau \right) :=\varphi ^{-1}\left( x_{0},\tau \right) $ are strictly monotone then for every $v\in L^{1,%
\text{ }\beta \left( x\right) }\left( \Omega \right) \cap L^{\psi
\left( x\right) }\left( \Omega \right) ,$ there exists an unique
$u\in S_{1,\gamma \left( x\right) ,\beta \left( x\right) ,\theta
\left( x\right) }\left(
\Omega \right) $ such that $u=\varphi ^{-1}\left( v\right) ,$ and for every $%
u\in S_{1,\gamma \left( x\right) ,\beta \left( x\right) ,\theta \left(
x\right) }\left( \Omega \right) ,$ there exists an unique $v\in L^{1,\text{ }%
\beta \left( x\right) }\left( \Omega \right) \cap L^{\psi \left( x\right)
}\left( \Omega \right) $ such that $v=\varphi \left( u\right) $ so that
shows the bijectivity of $\varphi$.
\end{proof}

\medskip Now, we give two important results of the Theorem 3.8 which help us
to understand the topology of the space $S_{1,\gamma \left( x\right) ,\beta
\left( x\right) ,\theta \left( x\right) }\left( \Omega \right) .$

\begin{corollary}
Let $\beta ,$ $\gamma $ and $\psi $ satisfy the conditions of
Theorem 3.8, then \newline $S_{1,\gamma \left( x\right) ,\beta
\left( x\right) ,\theta \left( x\right) }\left( \Omega \right) $
is a complete metric space with the metric which is defined below:
$\forall u,$ $v\in S_{1,\gamma \left( x\right) ,\beta \left(
x\right)
,\theta \left( x\right) }\left( \Omega \right) $%
\begin{equation*}
d_{S_{1}}\left( u,v\right) :=\left\Vert \varphi \left( u\right) -\varphi
\left( v\right) \right\Vert _{L^{\psi \left( x\right) }\left( \Omega \right)
}+\sum_{i=1}^{n}\left\Vert \varphi _{t}^{\prime }\left( u\right)
D_{i}u-\varphi _{t}^{\prime }\left( v\right) D_{i}u\right\Vert _{L^{\beta
\left( x\right) }\left( \Omega \right) },
\end{equation*}%
here $\varphi \left( x,t\right) =\left\vert t\right\vert ^{\frac{\gamma
\left( x\right) }{\beta \left( x\right) }}t$ and for every fixed $x\in\Omega$%
, $\varphi _{t}^{\prime }\left( t\right) =\left( \frac{\gamma \left(
x\right) }{\beta \left( x\right) }+1\right) \left\vert t\right\vert ^{\frac{%
\gamma \left( x\right) }{\beta \left( x\right) }}.$
\end{corollary}

\begin{corollary}
Under the conditions of Corollary 3.9, $\varphi $ is a
homeomorphism between the spaces $S_{1,\gamma \left( x\right)
,\beta \left(
x\right) ,\theta \left( x\right) }\left( \Omega \right) $ and $L^{1,\text{ }%
\beta \left( x\right) }\left( \Omega \right) \cap L^{\psi \left( x\right)
}\left( \Omega \right) .$
\end{corollary}

\begin{proof}
\text{(Sketch of the proof)} Since we have showed that $\varphi $ is a
bijection between $S_{1,\gamma \left( x\right) ,\beta \left( x\right)
,\theta \left( x\right) }\left( \Omega \right) $ and $L^{1,\text{ }\beta
\left( x\right) }\left( \Omega \right) \cap L^{\psi \left( x\right) }\left(
\Omega \right) ,$ it is sufficient to prove the continuity of $\varphi $ as
well as $\varphi ^{-1}$ in the sense of topology induced by the metric $%
d_{S_{1}}\left( .,.\right) .$ For this, we need to show that \newline
\textbf{(a)} $d_{S_{1}}\left( u_{m},u_{0}\right) \underset{m\nearrow \infty }%
{\longrightarrow }0\Rightarrow \varphi \left( u_{m}\right) {\underset{%
m\nearrow \infty }{\overset{L^{1,\text{ }\beta \left( x\right) }\left(
\Omega \right) \cap L^{\psi \left( x\right) }\left( \Omega \right) }{%
\longrightarrow }}\varphi \left( u_{0}\right) }$ for every \newline
$\left\{ u_{m}\right\} _{m=1}^{\infty }\in S_{1,\gamma \left( x\right)
,\beta \left( x\right) ,\theta \left( x\right) }\left( \Omega \right) $
which converges to $u_{0}$ and \medskip \newline
\textbf{(b)} $v_{m}{\underset{m\nearrow \infty }{\overset{L^{1,\text{ }\beta
\left( x\right) }\left( \Omega \right) \cap L^{\psi \left( x\right) }\left(
\Omega \right) }{\longrightarrow }}}v_{0}\Rightarrow d_{S_{1}}\left( \varphi
^{-1}\left( v_{m}\right) ,\varphi ^{-1}\left( v_{0}\right) \right) \underset{%
m\nearrow \infty }{\longrightarrow }0$ for every $\left\{ v_{m}\right\}
_{m=1}^{\infty }\in L^{1,\text{ }\beta \left( x\right) }\left( \Omega
\right) \cap L^{\psi \left( x\right) }\left( \Omega \right) $ which
converges to $v_{0}.$

\noindent Since for every $v_{m}$ and $v_{0}$, there exist unique $u_{m}$
and $u_{0}\in S_{1,\gamma \left( x\right) ,\beta \left( x\right) ,\theta
\left( x\right) }\left( \Omega \right) $ such that $\varphi \left(
u_{m}\right) =v_{m}$ and $\varphi \left( u_{0}\right) =v_{0},$ the
implication \textbf{(b)} can be written equivalently, \newline
$\varphi \left( u_{m}\right) {\underset{m\nearrow \infty }{\overset{L^{1,%
\text{ }\beta \left( x\right) }\left( \Omega \right) \cap L^{\psi \left(
x\right) }\left( \Omega \right) }{\longrightarrow }}\varphi }\left(
u_{0}\right) \Rightarrow d_{S_{1}}\left( u_{m},u_{0}\right) \underset{%
m\nearrow \infty }{\longrightarrow }0$ for every $\left\{
u_{m}\right\}\in S_{1,\gamma \left( x\right) ,\beta \left(
x\right) ,\theta \left( x\right) }\left( \Omega \right) $ which
converges to $u_{0}.$ Since the proofs of \textbf{(a)} and
\textit{\textbf{(b)}} are similar, we only prove
\textbf{(b)}: Let $v_{0},$ $\left\{ v_{m}\right\} _{m=1}^{\infty }\in L^{1,%
\text{ }\beta \left( x\right) }\left( \Omega \right) \cap L^{\psi \left(
x\right) }\left( \Omega \right) $ and $v_{m}\overset{L^{1,\text{ }\beta
\left( x\right) }\left( \Omega \right) \cap L^{\psi \left( x\right) }\left(
\Omega \right) }{{\longrightarrow }}v_{0}\Leftrightarrow \varphi \left(
u_{m}\right) \overset{L^{1,\text{ }\beta \left( x\right) }\left( \Omega
\right) \cap L^{\psi \left( x\right) }\left( \Omega \right) }{{%
\longrightarrow }}{\varphi }\left( u_{0}\right) .$ \newline To
verify $d_{S_{1}}\left( u_{m},u_{0}\right) \rightarrow 0,$ by
definition of $d_{S_{1}}$ it is sufficient to establish that
\begin{equation*}
\left\Vert \varphi _{t}^{\prime }\left( u_{m}\right) D_{i}u_{m}-\varphi
_{t}^{\prime }\left( u_{0}\right) D_{i}u_{0}\right\Vert _{L^{\beta \left(
x\right) }\left( \Omega \right) }\rightarrow 0\text{ and }\left\Vert \varphi
\left( u_{m}\right) -\varphi \left( u_{0}\right) \right\Vert _{L^{\psi
\left( x\right) }\left( \Omega \right) }\rightarrow 0
\end{equation*}%
as $m\nearrow \infty .$ The second convergence above is obvious by
definition of $d_{S_{1}}$ and the first one can be proved by
applying Theorem 3.8 and Vitali convergence theorem by virtue of
the equivalence $\left\Vert \varphi _{t}^{\prime }\left(
u_{m}\right) D_{i}u_{m}-\varphi _{t}^{\prime }\left( u_{0}\right)
D_{i}u_{0}\right\Vert _{L^{\beta \left( x\right) }\left( \Omega
\right) }\rightarrow 0\text{ }\Leftrightarrow \sigma _{\beta
}\left( \varphi _{t}^{\prime }\left( u_{m}\right)
D_{i}u_{m}-\varphi _{t}^{\prime }\left( u_{0}\right)
D_{i}u_{0}\right) \rightarrow 0 $.
\end{proof}

\begin{theorem}
Suppose that conditions of Theorem 3.8 are satisfied. Let $p\in
M_{0}\left( \Omega \right) $ and additionally $\beta $ satisfies
$1\leq \beta ^{-}\leq \beta \left( x\right) <n,$ $x\in \Omega $.
Assume that for $\varepsilon >0,$ the inequality
\begin{equation*}
p\left( x\right) +\varepsilon <\tfrac{n\left( \gamma \left( x\right) +\beta
\left( x\right) \right) }{n-\beta \left( x\right) },\text{ }x\in \Omega
\end{equation*}%
holds. Then we have the compact embedding%
\begin{equation*}
S_{1,\gamma \left( x\right) ,\beta \left( x\right) ,\theta \left( x\right)
}\left( \Omega \right) \hookrightarrow L^{p\left( x\right) }\left( \Omega
\right)
\end{equation*}
\end{theorem}

\begin{proof}
First, we show that $S_{1,\gamma \left( x\right) ,\beta \left(
x\right) ,\theta \left( x\right) }\left( \Omega \right) \subset
L^{p\left( x\right) }\left( \Omega \right) ,$ after that we prove
the compactness of this embedding.

For every $u\in S_{1,\gamma \left( x\right) ,\beta \left( x\right) ,\theta
\left( x\right) }\left( \Omega \right) ,$ by Theorem 3.8%
\begin{equation*}
\varphi \left( u\right) =\left\vert u\right\vert ^{\frac{\gamma
\left( x\right) }{\beta \left( x\right) }}u=v\in L^{1,\text{
}\beta \left( x\right) }\left( \Omega \right) \cap L^{\psi \left(
x\right) }\left( \Omega \right).
\end{equation*}%
Since $L^{1,\text{ }\beta \left( x\right) }\left( \Omega \right)
\cap L^{\psi \left( x\right) }\left( \Omega \right) \subset
W^{1,\text{ }\beta \left( x\right) }\left( \Omega \right)$ and the
embedding [8] $W^{1,\text{ }\beta \left( x\right) }\left( \Omega
\right) \subset L^{\beta ^{\ast }\left( x\right)
}\left( \Omega \right)$ exists for $\beta ^{\ast }\left( x\right) =\frac{n\beta \left( x\right) }{%
n-\beta \left( x\right) },$ therefore, we get that $v\in L^{\beta
^{\ast }\left( x\right) }\left( \Omega \right).$ So from the
definition of $v$ and the space $L^{\beta ^{\ast }\left( x\right)
}\left( \Omega \right) $ we
attain%
\begin{equation*}
v\in L^{\beta ^{\ast }\left( x\right) }\left( \Omega \right) \Leftrightarrow
u\in L^{\frac{n\left( \gamma \left( x\right) +\beta \left( x\right) \right)
}{n-\beta \left( x\right) }}\left( \Omega \right) ,
\end{equation*}%
As, by the conditions of theorem
\begin{equation*}
L^{\frac{n\left( \gamma \left( x\right) +\beta \left( x\right) \right) }{%
n-\beta \left( x\right) }}\left( \Omega \right) \subset L^{p\left(
x\right) +\varepsilon }\left( \Omega \right) \subset L^{p\left(
x\right) }\left( \Omega \right),
\end{equation*}%
thus, $u\in L^{p\left( x\right) }\left( \Omega \right) .$
\par Now
let us prove that this embedding is compact. \newline Let $\left\{
u_{m}\right\} _{m=1}^{\infty }\in S_{1,\gamma \left( x\right)
,\beta \left( x\right) ,\theta \left( x\right) }\left( \Omega
\right) $ be bounded sequence (i.e. $[u_{m}]_{\gamma ,\beta
,\theta }<\infty ,$ $\forall m\geq 1$).

\noindent From Theorem 3.8, we have
\begin{equation*}
\left\{ \varphi \left( u_{m}\right) \right\} =\left\{ v_{m}\right\}
_{m=1}^{\infty }\in W^{1,\text{ }\beta \left( x\right) }\left( \Omega
\right) ,
\end{equation*}%
since we have the compact embedding [8]%
\begin{equation*}
W^{1,\text{ }\beta \left( x\right) }\left( \Omega \right) \hookrightarrow
L^{q\left( x\right) }\left( \Omega \right)
\end{equation*}%
where $q\left( x\right) <\beta ^{\ast }\left( x\right) -\tilde{\varepsilon},$
$x\in \Omega $ and $\tilde{\varepsilon}\in \left( 0,n^{\prime }\right) $ ($%
n^{\prime }=\frac{n}{n-1}$). Thus, there exists a subsequence $\left\{
v_{m_{j}}\right\} \subset \left\{ v_{m}\right\} ,$ such that%
\begin{equation}
v_{m_{j}}\overset{L^{q\left( x\right) }\left( \Omega \right) }{%
\longrightarrow }v_{0}  \tag{3.7}
\end{equation}%
hence by (3.7), we obtain
\begin{equation*}
v_{m_{j}}\overset{a.e.}{\underset{\Omega }{\longrightarrow }}v_{0}
\end{equation*}%
as from Lemma 3.6, $\varphi ^{-1}(x,\tau )=\left\vert \tau \right\vert ^{%
\frac{\gamma \left( x\right) }{\gamma \left( x\right) +\beta \left( x\right)
}}\tau $ is continuous (with respect to $\tau $ and $x$) so we have
\begin{equation*}
\varphi ^{-1}\left( v_{m_{j}}\right) \overset{a.e.}{\underset{\Omega }{%
\longrightarrow }}\varphi ^{-1}\left( v_{0}\right) .
\end{equation*}%
To end the proof, we use Lemma 3.12\footnote{\noindent
\textbf{Lemma 3.12} Let $\Lambda $ be a family of real functions
defined on bounded
domain $\Omega .$ If there is an increasing function $\Phi :\left[ 0,\text{ }%
\infty \right) \rightarrow \left[ 0,\text{ }\infty \right) $ that
satisfies
\begin{equation*}
\lim_{t\rightarrow +\infty }\Phi \left( t\right) =+\infty
\end{equation*}%
and there is a positive constant $L$ such that%
\begin{equation*}
\int\limits_{\Omega }\left\vert f_{\alpha }\left( x\right)
\right\vert \Phi
\left( \left\vert f_{\alpha }\left( x\right) \right\vert \right) dx\leq L,%
\text{ \ \ }\forall f_{\alpha }\in \Lambda ,
\end{equation*}%
then every function in $\Lambda $ is Lebesgue integrable, and the
functions family $\Lambda $ possesses absolutely equicontinuous
integrals on $\Omega .$} [16, Theorem 7].\par Denote $u_{0}:=\varphi ^{-1}\left( v_{0}\right) $ and the set%
\begin{equation*}
\Lambda :=\left\{ f_{j}\text{ }|\text{ }f_{j}\left( x\right)
=\left\vert u_{m_{j}}\left( x\right) -u_{0}\left( x\right)
\right\vert ^{p\left( x\right) }\right\}
\end{equation*}%
and the function%
\begin{equation*}
\Phi \left( t\right) :=t^{\bar{\varepsilon}},\text{ }t\geq 0,\text{ }\bar{%
\varepsilon}=\frac{\varepsilon }{p^{+}}\text{ }\text{.}
\end{equation*}%
Clearly $\Phi :[0,\infty )\rightarrow \lbrack 0,\infty )$ is increasing and $%
\underset{t\rightarrow +\infty }{\lim }\Phi \left( t\right) =+\infty $

Furthermore for every $f_{j}\in \Lambda $ we have,%
\begin{equation*}
\int\limits_{\Omega }\left\vert f_{j}\left( x\right) \right\vert
\Phi \left( \left\vert f_{j}\left( x\right) \right\vert \right)
dx=\int\limits_{\Omega }\left\vert u_{m_{j}}-u_{0}\right\vert
^{p\left( x\right) }\left\vert u_{m_{j}}-u_{0}\right\vert
^{\bar{\varepsilon}p\left( x\right) }dx=
\end{equation*}%
\begin{equation*}
=\int\limits_{\Omega }\left\vert u_{m_{j}}-u_{0}\right\vert ^{\left( \bar{%
\varepsilon}+1\right) p\left( x\right) }dx\text{ }\text{.}
\end{equation*}%
Estimating the last integral by using Lemma 3.1, we arrive at%
\begin{equation*}
\leq \int\limits_{\Omega }\left\vert u_{m_{j}}-u_{0}\right\vert ^{p\left(
x\right) +\varepsilon }dx+\left\vert \Omega \right\vert
\end{equation*}%
using the well known inequality for absolute value above, we get%
\begin{equation}
\leq 2^{p^{+}+\varepsilon -1}\left( \int\limits_{\Omega }\left\vert
u_{m_{j}}\right\vert ^{p\left( x\right) +\varepsilon }dx+\int\limits_{\Omega
}\left\vert u_{0}\right\vert ^{p\left( x\right) +\varepsilon }dx\right)
+\left\vert \Omega \right\vert .  \tag{3.8}
\end{equation}%
Since $u_{0},\left\{ u_{m_{j}}\right\} \subset L^{p\left( x\right)
+\varepsilon }\left( \Omega \right) $ is bounded, from (3.8), there exists a
number $L>0$ such that%
\begin{equation}
\int\limits_{\Omega }\left\vert f_{j}\left( x\right) \right\vert
\Phi \left( \left\vert f_{j}\left( x\right) \right\vert \right)
dx\leq L,  \tag{3.9}
\end{equation}%
here $L=L\left( \left\vert \Omega \right\vert ,p^{+},\varepsilon ,\left\Vert
u_{0}\right\Vert _{L^{p\left( x\right) +\varepsilon }\left( \Omega \right)
},\left\Vert u_{m_{j}}\right\Vert _{L^{p\left( x\right) +\varepsilon }\left(
\Omega \right) }\right) .$

Consequently by (3.9), we obtain that the family of functions
$\Lambda $ possesses absolutely equicontinuous integrals on
$\Omega .$ Hence using
this and $u_{m_{j}}\overset{a.e.}{\underset{\Omega }{\longrightarrow }}%
u_{0}, $ we have [16]%
\begin{equation*}
\int\limits_{\Omega }\left\vert u_{m_{j}}\left( x\right) -u_{0}\left(
x\right) \right\vert ^{p\left( x\right) }dx\longrightarrow 0,\text{ }%
m_{j}\nearrow \infty ,
\end{equation*}%
that implies $\left\Vert u_{m_{j}}-u_{0}\right\Vert _{L^{p\left(
x\right) }\left( \Omega \right) }{\rightarrow }0$, so the proof is
complete.
\end{proof}

\section{Proof of The Existence Theorem}

\noindent The proof is based on Theorem 2.5. We introduce the following
spaces and mappings in order to apply Theorem 2.5 to prove Theorem 2.4.%
\begin{equation*}
S_{gY_{0}}:=\mathring{S}_{1,q_{0}\left( x\right) \left( p_{0}\left( x\right)
-2\right) ,q_{0}\left( x\right) ,\alpha \left( x\right) }\left( \Omega
\right) ,\text{ }X:=W_{0}^{1,\text{ }p_{1}\left( x\right) }\left( \Omega
\right) ,
\end{equation*}%
\begin{equation*}
Y:=W^{-1,\text{ }q_{0}\left( x\right) }\left( \Omega \right) +L^{\alpha
^{\prime }\left( x\right) }\left( \Omega \right) ,
\end{equation*}%
\begin{equation*}
X_{0}:=W_{0}^{1,\text{ }p_{0}\left( x\right) }\left( \Omega \right) \cap
W_{0}^{1,\text{ }p_{1}\left( x\right) }\left( \Omega \right) \cap L^{\alpha
\left( x\right) }\left( \Omega \right) \text{ and}
\end{equation*}%
\begin{equation*}
Y_{2_{1}}:=W^{-1,\text{ }2}\left( \Omega \right) ,\text{ \ }%
Y_{2_{2}}:=L^{2}\left( \Omega \right) \text{ }
\end{equation*}%
and
\begin{equation}
A\left( u\right) :=-div\left(\left\vert\nabla
u\right\vert^{p_{1}\left( x\right)-2}\nabla u\right) ,  \tag{4.1}
\end{equation}%
\begin{equation}
B_{1}\left( u\right) :=-\sum_{i=1}^{n}D_{i}\left( \left\vert u\right\vert
^{p_{0}\left( x\right) -2}D_{i}u\right) ,\text{ \ }B_{2}\left( u\right)
:=c\left( x,u\right) ,  \tag{4.2}
\end{equation}%
\begin{equation}
B:=B_{1}+B_{2}\text{ and }T:=A+B.  \tag{4.3}
\end{equation}%
We show that all the conditions of Theorem 2.5 are satisfied by
proving some lemmas. Then based on these lemmas, we establish the
proof of Theorem 2.4.

\begin{lemma}
Under the conditions of Theorem 2.4, the operator $T$ defined by (4.3) is
coercive in the generalized sense on $X_{0}.$
\end{lemma}

\begin{proof}
For every $u\in W_{0}^{1,\text{ }p_{0}\left( x\right) }\left( \Omega \right)
\cap W_{0}^{1,\text{ }p_{1}\left( x\right) }\left( \Omega \right) \cap
L^{\alpha \left( x\right) }\left( \Omega \right) ,$ we have
\begin{equation*}
\langle T\left( u\right) ,u\rangle =\langle A\left( u\right) ,u\rangle
+\langle B\left( u\right) ,u\rangle
\end{equation*}%
\begin{equation}
=\int\limits_{\Omega }\left\vert \nabla u\right\vert ^{p_{1}\left(
x\right) }dx+\sum_{i=1}^{n}\int\limits_{\Omega }\left\vert
u\right\vert ^{p_{0}\left( x\right) -2}\left\vert
D_{i}u\right\vert ^{2}dx+\int\limits_{\Omega }c\left( x,u\right)
udx.  \tag{4.4}
\end{equation}%
If we take account the condition \emph{(U2)} into the third
integral of (4.4) and apply the following simple calculated
inequality
\begin{align*}
&\int\limits_{\Omega }\left\vert u\right\vert ^{\alpha \left(
x\right) }dx+\sum_{i=1}^{n}\int\limits_{\Omega }\left\vert
u\right\vert ^{q_{0}\left( x\right) \left( p_{0}\left( x\right)
-2\right) }\left\vert D_{i}u\right\vert ^{q_{0}\left( x\right)
}dx\\
&\leq \left( n+1\right) \left( \int\limits_{\Omega }\left\vert
u\right\vert ^{\alpha \left( x\right)
}dx+\sum_{i=1}^{n}\int\limits_{\Omega }\left\vert u\right\vert
^{p_{0}\left( x\right) -2}\left\vert D_{i}u\right\vert
^{2}dx+\left\vert \Omega \right\vert \right)  \tag{4.5}
\end{align*}
we get,
\begin{align*}
\langle T\left( u\right) ,u\rangle &\geq C_{8}\left(
\int\limits_{\Omega }\left\vert u\right\vert ^{\alpha \left(
x\right) }dx+\sum_{i=1}^{n}\int\limits_{\Omega }\left\vert
u\right\vert ^{q_{0}\left( x\right) \left( p_{0}\left( x\right)
-2\right) }\left\vert D_{i}u\right\vert ^{q_{0}\left( x\right)
}\right)\notag \\ &+\int\limits_{\Omega }\left\vert \nabla
u\right\vert ^{p_{1}\left( x\right) }dx-C_{9},  \tag{4.6}
\end{align*}%
here $C_{8}=C_{8}\left( n,\tilde{C}\right) >0$ and $C_{9}=$
$C_{9}\left( n,\left\vert \Omega \right\vert ,\tilde{C}\right) >0$
are constants.\newline Applying Lemma 3.4 to estimate the
right-hand side of the inequality (4.6), we obtain
\begin{equation}
\langle T\left( u\right) ,u\rangle \geq C_{10}\left(
[u]_{S_{q_{0}\left( p_{0}-2\right) ,q_{0},\alpha
}}^{q_{0}^{-}+1}+\left\Vert u\right\Vert _{W_{0}^{1,\text{
}p_{1}\left( x\right)}\left( \Omega \right)}^{p_{1}^{-}}\right)
-C_{11}, \tag{4.7}
\end{equation}%
by the definitions of $ [.]_{S_{q_{0}\left( p_{0}-2\right)
,q_{0},\alpha }}$ and $\left\Vert .\right\Vert _{W_{0}^{1,\text{ }p_{1}\left( x\right) }\left( \Omega \right)}.$ \par Since $q_{0}^{-}+1>2$ and $p_{1}^{-}>1$ thus, $\lambda _{0}\left( \tau \right) =\tau ^{q_{0}^{-}}$ and $%
\lambda _{1}\left( \tau \right) =\tau ^{p_{1}^{-}-1}$ tends to
infinity when $\tau \nearrow \infty ,$ (see Theorem 2.5) so that
means operator $T$ is coercive in the generalized sense on
$X_{0}.$
\end{proof}

\begin{lemma}
Under the conditions of Theorem 2.4, the operator $A$ is monotone
and bounded from
$W_{0}^{1,\text{ }p_{1}\left( x\right) }\left( \Omega \right) $ into $W^{-1,%
\text{ }q_{1}\left( x\right) }\left( \Omega \right) .$
\end{lemma}

\begin{proof}
First we prove that $A:$ $W_{0}^{1,\text{ }p_{1}\left( x\right)
}\left( \Omega \right) $ $\rightarrow W^{-1,\text{ }q_{1}\left(
x\right) }\left(
\Omega \right) $ is bounded. For this, it is sufficient to investigate the dual form $\langle A\left( u\right) ,v\rangle $ for every $%
v\in W_{0}^{1,\text{ }p_{1}\left( x\right) }\left( \Omega \right) ,$%
\begin{equation*}
\left\vert \langle A\left( u\right) ,v\rangle \right\vert
=\left\vert \int\limits_{\Omega }\left\vert\nabla u\right\vert
^{p_{1}\left( x\right) -2}\nabla u\cdot\nabla vdx\right\vert
\end{equation*}
using the generalized H\"{o}lder inequality to the right hand side of the
above equation, we get
\begin{equation}
\leq 2\left\Vert \left\vert\nabla u\right\vert ^{p_{1}\left(
x\right) -1}\right\Vert_{L^{q_{1}\left(
x\right)}\left( \Omega \right) }\left\Vert v\right\Vert _{W_{0}^{1,%
\text{ }p_{1}\left( x\right) }\left( \Omega \right) }.  \tag{4.8}
\end{equation}%

Thus by (4.8) we demonstrate the boundedness of $A$ from
$W_{0}^{1,\text{ }p_{1}\left( x\right) }\left( \Omega \right) $ to
$W^{-1,\text{ }q_{1}\left( x\right) }\left( \Omega \right) .$

\noindent Now let us show that $A:$ $W_{0}^{1,\text{ }p_{1}\left( x\right)
}\left( \Omega \right) $ $\rightarrow W^{-1,\text{ }q_{1}\left( x\right)
}\left( \Omega \right) $ is a monotone operator.

Indeed for every $u,$ $v\in W_{0}^{1,\text{ }p_{1}\left( x\right) }\left(
\Omega \right) $ we have,%
\begin{align*}
&\langle A\left( u\right) -A\left( v\right) ,u-v\rangle= \\
&=\int\limits_{\Omega }\left(\left\vert\nabla u\right\vert
^{p_{1}\left( x\right) -2}\nabla u-\left\vert\nabla v\right\vert
^{p_{1}\left( x\right) -2}\nabla v\right)\cdot\left( \nabla
u-\nabla v\right) dx.
\end{align*}%
Since the inequality $ \left( \left\vert a\right\vert
^{p-2}a-\left\vert
b\right\vert ^{p-2}b\right)\cdot\left( a-b\right)\geq 0 $ is valid for $1< p<\infty,$ $a,$ $%
b\in \mathbb{R}^{n}$ from the last equality, we attain
\begin{equation*}
\langle A\left( u\right) -A\left( v\right) ,u-v\rangle \geq 0
\end{equation*}%
which completes the proof. \footnote{%
Here, we note that since $A$ is monotone and hemicontinuous then
it is pseudo-monotone [25]}.
\end{proof}

\begin{lemma}
Under the conditions of Theorem 2.4, $B$ is a bounded operator from $%
\mathring{S}_{1,q_{0}\left( x\right) \left( p_{0}\left( x\right) -2\right)
,q_{0}\left( x\right) ,\alpha \left( x\right) }\left( \Omega \right) $ into $%
W^{-1,\text{ }q_{0}\left( x\right) }\left( \Omega \right) +L^{\alpha
^{\prime }\left( x\right) }\left( \Omega \right) .$
\end{lemma}

\begin{proof}
Since $B=B_{1}+B_{2},$ we shall show that both $B_{1}$ and $B_{2}$ are
bounded.

First let us verify $B_{2}:$ $\mathring{S}_{1,q_{0}\left( x\right) \left(
p_{0}\left( x\right) -2\right) ,q_{0}\left( x\right) ,\alpha \left( x\right)
}\left( \Omega \right) $ $\rightarrow W^{-1,\text{ }q_{0}\left( x\right)
}\left( \Omega \right) +L^{\alpha ^{\prime }\left( x\right) }\left( \Omega
\right) $ is bounded:

As $\mathring{S}_{1,q_{0}\left( x\right) \left( p_{0}\left( x\right)
-2\right) ,q_{0}\left( x\right) ,\alpha \left( x\right) }\left( \Omega
\right) \subset L^{\alpha \left( x\right) }\left( \Omega \right) ,$ it is
sufficient to show the boundedness of $B_{2},$ from $L^{\alpha \left(
x\right) }\left( \Omega \right) $ to $L^{\alpha ^{\prime }\left( x\right)
}\left( \Omega \right) .$

For every $u\in L^{\alpha \left( x\right) }\left( \Omega \right) $%
\begin{equation*}
\sigma _{\alpha ^{\prime }}\left( B_{2}\left( u\right) \right)
=\int\limits_{\Omega }\left\vert B_{2}\left( u\right) \right\vert
^{\alpha ^{\prime }\left( x\right) }dx=\int\limits_{\Omega
}\left\vert c\left( x,u\right) \right\vert ^{\alpha ^{\prime
}\left( x\right) }dx,
\end{equation*}%
here taking the conditions of Theorem 2.4 into account and
estimating the above integral, we obtain%
\begin{equation}
\sigma _{\alpha ^{\prime }}\left( B_{2}\left( u\right) \right)\leq
2\left( \left\Vert c_{0}\right\Vert _{L^{\infty }\left( \Omega
\right) }^{2}\sigma _{\alpha }\left( u\right) +\sigma _{\alpha
^{\prime }}\left( c_{1}\right) \right).  \tag{4.9}
\end{equation}%
So from (4.9), we arrive at $B_{2}$ is bounded.

Now let us prove that $B_{1}:$ $\mathring{S}_{1,q_{0}\left(
x\right) \left( p_{0}\left( x\right) -2\right) ,q_{0}\left(
x\right) ,\alpha \left( x\right) }\left( \Omega \right) $
$\rightarrow W^{-1,\text{ }q_{0}\left( x\right) }\left( \Omega
\right) +L^{\alpha ^{\prime }\left( x\right) }\left( \Omega
\right) $ is bounded: $\forall i=\overline{1,n},$ denote
$b_{i}\left( x\right) :=\left\vert u\right\vert ^{p_{0}\left(
x\right) -2}\left\vert D_{i}u\left( x\right) \right\vert ,$ for every $%
v\in W_{0}^{1,\text{ }p_{0}\left( x\right) }\left( \Omega \right)
$
\begin{equation*}
\left\vert \langle B_{1}\left( u\right) ,v\rangle \right\vert =\left\vert
-\sum_{i=1}^{n}\int\limits_{\Omega }D_{i}\left( \left\vert u\right\vert
^{p_{0}\left( x\right) -2}D_{i}u\right) vdx\right\vert,
\end{equation*}
applying the generalized H\"{o}lder inequality to the right hand side of the
above equation, we arrive at
\begin{equation*}
\left\vert \langle B_{1}\left( u\right) ,v\rangle \right\vert\leq
2\left( \sum_{i=1}^{n}\left\Vert b_{i}\right\Vert _{L^{q_{0}\left(
x\right) }\left( \Omega \right) }\right) \left\Vert v\right\Vert _{W_{0}^{1,%
\text{ }p_{0}\left( x\right) }\left( \Omega \right) }.
\end{equation*}
Since $u\in \mathring{S}_{1,q_{0}\left( x\right) \left(
p_{0}\left( x\right) -2\right) ,q_{0}\left( x\right) ,\alpha
\left( x\right) }\left( \Omega \right) ,$ from (2.3) and the
definition of the functions $b_{i}\left( x\right)$, obviously
$\sum\limits_{i=1}^{n}\left\Vert b_{i}\right\Vert _{L^{q_{0}\left(
x\right) }\left( \Omega \right) }<\infty .$ Thus we verify that
$B_{1}$ is bounded. \newline
Consequently, we prove that $B$ is a bounded operator from $\mathring{S}%
_{1,q_{0}\left( x\right) \left( p_{0}\left( x\right) -2\right) ,q_{0}\left(
x\right) ,\alpha \left( x\right) }\left( \Omega \right) $ to $W^{-1,\text{ }%
q_{0}\left( x\right) }\left( \Omega \right) +L^{\alpha ^{\prime }\left(
x\right) }\left( \Omega \right) .$
\end{proof}

\begin{lemma}
Under the conditions of Theorem 2.4, $B$ is a weakly compact operator from $%
\mathring{S}_{1,q_{0}\left( x\right) \left( p_{0}\left( x\right) -2\right)
,q_{0}\left( x\right) ,\alpha \left( x\right) }\left( \Omega \right) $ into $%
W^{-1,\text{ }q_{0}\left( x\right) }\left( \Omega \right) +L^{\alpha
^{\prime }\left( x\right) }\left( \Omega \right) .$
\end{lemma}

\begin{proof}
Since $B=B_{1}+B_{2},$ we shall show that both $B_{1}$ and $B_{2}$
are weakly compact. \\ First we show the weak compactness of
$B_{1}:$ Let $\left\{ u_{m}\right\} _{m=1}^{\infty },$ $u_{0}$
$\in \mathring{S}_{1,q_{0}\left( x\right) \left( p_{0}\left(
x\right) -2\right) ,q_{0}\left( x\right) }\left(
\Omega \right) \cap L^{\alpha \left( x\right) }\left( \Omega \right) $ and $%
u_{m}\overset{S_{gY_{0}}}{\rightharpoonup }u_{0}.$ By Theorem 3.8 we have
\begin{equation*}
\left\{ w_{m}\right\} _{m=1}^{\infty }:=\left\{ \varphi \left( u_{m}\right)
\right\} _{m=1}^{\infty }=\left\{ \left\vert u_{m}\right\vert ^{p_{0}\left(
x\right) -2}u_{m}\right\} _{m=1}^{\infty }\subset W_{0}^{1,\text{ }%
q_{0}\left( x\right) }\left( \Omega \right).
\end{equation*}%
As $q_{0}^{-}>1$ that implies $W_{0}^{1,\text{ }%
q_{0}\left( x\right) }\left( \Omega \right) $ is a reflexive space thus,
there exists a subsequence $\left\{ w_{m_{j}}\right\} _{j=1}^{\infty }$ of $%
\left\{ w_{m}\right\} $ such that
\begin{equation*}
w_{m_{j}}=\left\vert u_{m_{j}}\right\vert ^{p_{0}\left( x\right) -2}u_{m_{j}}%
\overset{W_{0}^{1,\text{ }q_{0}\left( x\right) }\left( \Omega \right) }{%
\rightharpoonup }\xi .
\end{equation*}%
Let us verify $\xi =\left\vert u_{0}\right\vert ^{p_{0}\left(
x\right) -2}u_{0}$. Since $W_{0}^{1,\text{ }q_{0}\left( x\right)
}\left( \Omega \right) \hookrightarrow L^{q_{0}\left( x\right)
}\left( \Omega \right) $ therefore there exist a subsequence
$\left\{ w_{m_{j_{k}}}\right\} \subset \left\{ w_{m_{j}}\right\} $
(denote
this subsequence by $w_{m_{j}}$ in order to avoid notation confusion) such that%
\begin{equation*}
\left\vert u_{m_{j}}\right\vert ^{p_{0}\left( x\right) -2}u_{m_{j}}\overset{%
L^{q_{0}\left( x\right) }\left( \Omega \right) }{\longrightarrow }\xi
\end{equation*}%
hence%
\begin{equation}
\varphi \left( u_{m_{j}}\right) =\left\vert u_{m_{j}}\right\vert
^{p_{0}\left( x\right) -2}u_{m_{j}}\overset{\Omega }{\underset{a.e.}{%
\longrightarrow }}\xi  \tag{4.10}
\end{equation}%
as from Lemma 3.6, $\varphi ^{-1}(x,\tau )=\left\vert \tau \right\vert ^{-%
\frac{p_{0}\left( x\right) -2}{p_{0}\left( x\right) -1}}\tau $ is continuous
(with respect to $\tau $ and $x$) so using (4.10) we obtain%
\begin{equation}
u_{m_{j}}\overset{\Omega }{\underset{a.e.}{\longrightarrow }}\varphi
^{-1}(x,\xi )=\varphi ^{-1}(\xi ),  \tag{4.11}
\end{equation}%
hence by (4.11), we arrive at $\varphi ^{-1}(\xi )=u_{0}$, equivalently $%
\xi =\left\vert u_{0}\right\vert ^{p_{0}\left( x\right) -2}u_{0}.$

To verify the weak compactness of
$B_{1}$, we must show that for arbitrary $v\in W_{0}^{1,\text{ }%
p_{0}\left( x\right) }\left( \Omega \right) $%
\begin{equation*}
\langle B_{1}\left( u_{m_{j}}\right) ,v\rangle \rightarrow \langle
B_{1}\left( u_{0}\right) ,v\rangle ,\text{ \ }j\nearrow \infty .
\end{equation*}%
By the definition of operator $B_{1},$%
\begin{align*}
\langle B_{1}\left( u_{m_{j}}\right) ,v\rangle
&=\sum_{i=1}^{n}\langle -D_{i}\left( \left\vert
u_{m_{j}}\right\vert ^{p_{0}\left( x\right)
-2}D_{i}u_{m_{j}}\right) ,v\rangle
\\&=\sum_{i=1}^{n}\langle
\left\vert u_{m_{j}}\right\vert ^{p_{0}\left( x\right)
-2}D_{i}u_{m_{j}},\text{ }D_{i}v\rangle. \tag{4.12}
\end{align*}
Using Lemma 3.6 and chain rule we have the following equality%
\begin{align*}
D_{i}\left( \left\vert u_{m_{j}}\right\vert ^{p_{0}\left( x\right)
-2}u_{m_{j}}\right) &=\left( p_{0}\left( x\right) -2\right)
\left\vert u_{m_{j}}\right\vert ^{p_{0}\left( x\right)
-2}D_{i}u_{m_{j}}\\ &+\left( D_{i}p_{0}\right) \left\vert
u_{m_{j}}\right\vert ^{p_{0}\left( x\right) -2}u_{m_{j}}\ln
\left\vert u_{m_{j}}\right\vert,  \tag{4.13}
\end{align*}%
if we insert the equality (4.13) into (4.12), we obtain%
\begin{align*}
\langle B_{1}\left( u_{m_{j}}\right) ,v\rangle
&=\sum_{i=1}^{n}\langle \left( \tfrac{1}{p_{0}\left( x\right)
-2}\right) D_{i}\left( \left\vert u_{m_{j}}\right\vert
^{p_{0}\left( x\right) -2}u_{m_{j}}\right) ,D_{i}v\rangle\\
&-\sum_{i=1}^{n}\langle \left( \tfrac{D_{i}p_{0}}{p_{0}\left( x\right) -2}%
\right) \left\vert u_{m_{j}}\right\vert ^{p_{0}\left( x\right)
-2}u_{m_{j}}\ln \left\vert u_{m_{j}}\right\vert ,D_{i}v\rangle .  \tag{4.14}
\end{align*}%
Let us denote the first sum in (4.14) by $I_{1}$ and the second one by $%
I_{2} $ i.e.%
\begin{equation*}
\langle B_{1}\left( u_{m_{j}}\right) ,v\rangle =I_{1}-I_{2},
\end{equation*}%
if we use the same manner in [20, Lemma 3.3] and pass to the limit in $%
I_{1}, $ we obtain%
\begin{equation}
I_{1}\underset{\text{\ }j\nearrow \infty }{\longrightarrow }%
\sum_{i=1}^{n}\langle \left( \tfrac{1}{p_{0}\left( x\right)
-2}\right) D_{i}\left( \left\vert u_{0}\right\vert ^{p_{0}\left(
x\right) -2}u_{0}\right) ,D_{i}v\rangle.  \tag{4.15}
\end{equation}%
Considering Lemma 3.2 together with Theorem 3.11 and continuity of
the function $\left\vert t\right\vert ^{p_{0}\left( x\right)
-2}t\ln \left\vert t\right\vert $ with respect to $t$ and pass to
the limit in $I_{2}$, we
obtain%
\begin{equation}
I_{2}\underset{\text{\ }j\nearrow \infty }{\longrightarrow }%
\sum_{i=1}^{n}\langle \left( \tfrac{D_{i}p_{0}}{p_{0}\left( x\right) -2}%
\right) \left\vert u_{0}\right\vert ^{p_{0}\left( x\right)
-2}u_{0}\ln \left\vert u_{0}\right\vert ,D_{i}v\rangle. \tag{4.16}
\end{equation}%
Hence from (4.15) and (4.16), we have%
\begin{align*}
&\langle B_{1}\left( u_{m_{j}}\right) ,v\rangle \underset{\text{\
}j\nearrow \infty }{\longrightarrow }\sum_{i=1}^{n}\langle \left(
\tfrac{1}{p_{0}\left( x\right) -2}\right) D_{i}\left( \left\vert
u_{0}\right\vert ^{p_{0}\left( x\right) -2}u_{0}\right)
,D_{i}v\rangle \\
&-\sum_{i=1}^{n}\langle \left( \tfrac{D_{i}p_{0}}{p_{0}\left( x\right) -2}%
\right) \left\vert u_{0}\right\vert ^{p_{0}\left( x\right)
-2}u_{0}\ln \left\vert u_{0}\right\vert ,D_{i}v\rangle \\
&=\sum_{i=1}^{n}\langle \tfrac{1}{p_{0}\left( x\right) -2}\left[
D_{i}\left( \left\vert u_{0}\right\vert ^{p_{0}\left( x\right)
-2}u_{0}\right) -\left( D_{i}p_{0}\right) \left\vert
u_{0}\right\vert ^{p_{0}\left( x\right) -2}u_{0}\ln \left\vert
u_{0}\right\vert \right] ,D_{i}v\rangle \\&
\text{thus by(4.13), we
have}\\
&=\sum_{i=1}^{n}\langle \left\vert u_{0}\right\vert ^{p_{0}\left(
x\right) -2}D_{i}u_{0},\text{ }D_{i}v\rangle=\langle B_{1}\left(
u_{0}\right) ,v\rangle.
\end{align*}%
Therefore, we prove the weak compactness of $B_{1}$ from $\mathring{S}%
_{1,q_{0}\left( x\right) \left( p_{0}\left( x\right) -2\right) ,q_{0}\left(
x\right) ,\alpha \left( x\right) }\left( \Omega \right) $ to $W^{-1,\text{ }%
q_{0}\left( x\right) }\left( \Omega \right) +L^{\alpha ^{\prime }\left(
x\right) }\left( \Omega \right) .$

Now we prove the weak compactness of $B_{2}.$ As $\left( \alpha
^{\prime }\right) ^{-}>1,$ $L^{\alpha ^{\prime }\left( x\right)
}\left( \Omega \right) $ is a reflexive space and $\left\{
B_{2}\left( u_{m}\right) \right\} _{m=1}^{\infty }:=\left\{ \eta
_{m}\right\} _{m=1}^{\infty }\subset L^{\alpha ^{\prime }\left(
x\right) }\left( \Omega \right)$ is bounded (see, Lemma 4.3), then
there exists a subsequence $\left\{ \eta _{m_{j}}\right\} \subset
\left\{ \eta _{m}\right\}
$ such that%
\begin{equation*}
\eta _{m_{j}}=B_{2}\left( u_{m_{j}}\right) \overset{L^{\alpha
^{\prime }\left( x\right) }\left( \Omega \right) }{\rightharpoonup
}\psi .
\end{equation*}%
By Theorem 3.11, the embedding%
\begin{equation}
\mathring{S}_{1,q_{0}\left( x\right) \left( p_{0}\left( x\right) -2\right)
,q_{0}\left( x\right) ,\alpha \left( x\right) }\left( \Omega \right)
\hookrightarrow L^{s\left( x\right) }\left( \Omega \right)  \tag{4.17}
\end{equation}%
is compact for $s\left( .\right) $ which satisfies the inequality $s\left( x\right) <%
\frac{np_{0}\left( x\right) }{n-q_{0}\left( x\right) },$ $x\in \Omega .$

Thus, by (4.17) there exists a subsequence $\left\{ u_{m_{j_{k}}}\right\}
\subset \left\{ u_{m_{j}}\right\} $ (let us denote this subsequence by $%
u_{m_{j}}$ in order to avoid notation confusion.) such that%
\begin{equation*}
u_{m_{j}}\overset{L^{s\left( x\right) }\left( \Omega \right) }{%
\longrightarrow }u_{0}
\end{equation*}%
so%
\begin{equation}
u_{m_{j}}\overset{\Omega }{\underset{a.e.}{\longrightarrow }}u_{0}.
\tag{4.18}
\end{equation}%
Since the function $c\left( x,\tau \right) $ is continuous with respect to
variable $\tau $ ($c\left( x,\tau \right) $ is Carath\`{e}dory function)$,$
by (4.18)%
\begin{equation}
B_{2}\left( u_{m_{j}}\right) =c\left( x,u_{m_{j}}\right) \overset{\Omega }{%
\underset{a.e.}{\longrightarrow }}c\left( x,u_{m_{j}}\right)
=B_{2}\left( u_{0}\right).  \tag{4.19}
\end{equation}%
Therefore, from (4.19) we obtain that $\psi =B_{2}\left( u_{0}\right) .$

Finally, we arrive at%
\begin{equation*}
B_{2}\left( u_{m_{j}}\right) \overset{L^{\alpha ^{\prime }\left( x\right)
}\left( \Omega \right) }{\rightharpoonup }B_{2}\left( u_{0}\right)
\end{equation*}%
which implies that%
\begin{equation}
B_{2}\left( u_{m_{j}}\right) \overset{W^{-1,\text{ }q_{0}\left(
x\right) }\left( \Omega \right) +L^{\alpha ^{\prime }\left(
x\right) }\left( \Omega \right) }{\rightharpoonup }B_{2}\left(
u_{0}\right).  \tag{4.20}
\end{equation}%
So, from (4.20) we obtain the weak compactness of $B_{2}$ which provides, as
a result, the weak compactness of the operator $B.$
\end{proof}

\medskip It now remains to define the corresponding operators for $B$ in condition "2)" of Theorem
2.5 to apply this theorem to the problem (1.1).
\\ Since $B=B_{1}+B_{2},$ according to condition "2)" we define
corresponding $B_{01}$ with regard to $B_{1}$ and corresponding
$B_{02}$
with regard to $B_{2}$ as below:%
\begin{equation*}
B_{01}\left( u\right) :=-\sum_{i=1}^{n}D_{i}\left( \left\vert u\right\vert ^{%
\frac{p_{0}\left( x\right) -2}{2}}D_{i}u\right)
\end{equation*}%
and%
\begin{equation*}
B_{02}\left( u\right) :=\left[ c\left( x,u\right) u\right] ^{\frac{1}{2}}.
\end{equation*}%
Note that here $c\left( x,\tau\right)\tau>0$ by the condition
\emph{(U2)}.

By the same arguments which are used in the proof of Lemma 4.3 to
establish the boundedness of the operators $B_{1}$ and $B_{2},$ we
can show that the operators $B_{01}$ and $B_{02}$ are bounded
between the spaces
which are introduced below:%
\begin{equation*}
B_{01}:X_{0}\subset \mathring{S}_{1,p_{0}\left( x\right) -2,2,\alpha \left(
x\right) }\left( \Omega \right) \longrightarrow W^{-1,\text{ }2}\left(
\Omega \right)
\end{equation*}%
and%
\begin{equation*}
B_{02}:X_{0}\subset \mathring{S}_{1,q_{0}\left( x\right) \left(
p_{0}\left( x\right) -2\right) ,q_{0}\left( x\right) ,\alpha
\left( x\right) }\left( \Omega \right) \longrightarrow L^{2}\left(
\Omega \right).
\end{equation*}

\noindent Here, we have to prove the weak compactness of $B_{01}$
and $B_{02}$ to show that condition "2)" in Theorem 2.5 is
satisfied.
\par
By using the similar manner which has been established in the
proof of Lemma 4.4 and by the definition of the functionals
corresponding to operators $B_{01}$ and $B_{02}$ (see (4.21),
(4.22)), following lemmas can be proved straightforwardly, so we
omit the proofs of them.

\begin{lemma}
Under the conditions of Theorem 2.4, $B_{01}$ is weakly compact operator
from $X_{0}$ into $W^{-1,\text{ }2}\left( \Omega \right) .$ Moreover for the
function $\mu \left( \tau \right) =\tau ^{2}$ and for every $u\in X_{0},$
the equality%
\begin{equation}
\langle B_{1}\left( u\right) ,u\rangle \equiv \mu \left( \left\Vert
B_{01}\left( u\right) \right\Vert _{W^{-1,\text{ }2}\left( \Omega \right)
}\right)  \tag{4.21}
\end{equation}%
holds.
\end{lemma}

\begin{lemma}
Under the conditions of Theorem 2.4, $B_{02}$ is weakly compact operator
from $X_{0}$ into $L^{2}\left( \Omega \right) .$ Moreover for every $u\in
X_{0},$ the equality%
\begin{equation}
\langle B_{2}\left( u\right) ,u\rangle \equiv \mu \left( \left\Vert
B_{02}\left( u\right) \right\Vert _{L^{2}\left( \Omega \right) }\right)
\tag{4.22}
\end{equation}%
holds.
\end{lemma}

\noindent Now we can give the proof of Theorem 2.4.\medskip

\begin{proof}
\textbf{(Proof of Theorem 2.4)} In Lemmas 4.1-4.6, we show that
all the conditions of Theorem 2.5 are satisfied for problem (1.1)
under the conditions of Theorem 2.4. Consequently, we establish
that Theorem 2.5 can be applied to the problem (1.1). Hence using
this theorem, we obtain the existence of a weak solution of
problem (1.1) in the sense of Definition 2.3.
\end{proof}


\begin{thebibliography}{99}

\bibitem{} E. Acerbi, G. Mingione: \emph{Regularity results for
stationary electro-rheological liquids}. Arch. Ration. Mech. Anal.
164 (2002), 213--259.

\bibitem{} R. A. Adams: \emph{Sobolev Spaces.} Academic Press, New York,
(1975).

\bibitem{} S. N. Antontsev, S. I. Shmarev: \emph{A model porous medium
equation with variable exponent of nonlinearity: existence,
uniqueness and localization properties of solutions}. Nonlinear
Anal. 60 (2005), 515--545.

\bibitem{} S. N. Antontsev, J. F. Rodrigues: \emph{On stationary
thermo-rheological viscous flows.} Ann. Univ. Ferrara, Sez. VII,
Sci. Mat. 52 (1) (2006), 19-36.

\bibitem{} S. N. Antontsev, S. I. Shmarev: \emph{On the localization of
solutions of elliptic equations with nonhomogeneous anisotropic degeneration}%
. (Russian) Sibirsk. Mat. Zh. 46., (2005), no. 5, 963--984; translation in
Siberian Math. J. 46., no. 5 (2005), 765--782.

\bibitem{} S. N. Antontsev, S. I. Shmarev: \emph{Elliptic equations and
systems with nonstandard growth conditions: Existence, uniqueness
and localization properties of solutions}. Nonlinear Anal. 65
(2006), 728-761.

\bibitem{} L. Diening: \emph{Theoretical and numerical results for
electrorheological fluids}. Ph.D. Thesis, (2002)

\bibitem{} L. Diening, P. Harjulehto, P. Hast\"{o}, M. Ru\v{z}i\v{c}ka:
\emph{Lebesgue and Sobolev Spaces with Variable Exponents}.
Lecture Notes in Mathematics, 2017. Springer, Heidelberg (2011).

\bibitem{} Y. Chen, S. Levine, M. Rao: \emph{Variable exponent, linear
growth functionals in image restoration}. SIAM J. Appl. Math. 66
(2006), 1383-1406.

\bibitem{} X. Fan, D. Zhao: \emph{On the spaces }$L^{p\left( x\right) }(%
\Omega
)$\emph{\ and }$W^{k,p\left( x\right) }(%
\Omega
)$\emph{. }J. Math. Anal. Appl. 263 (2001), 424--446.

\bibitem{} X. Fan, J. Shen, D. Zhao: \emph{Sobolev embedding theorems for
spaces }$W^{k,p\left( x\right) }(%
\Omega
)$. J. Math. Anal. Appl. 262, no. 2 (2001), 749--760.

\bibitem{} O. Kovacik, J. Rakosnik: \emph{On spaces }$L^{p\left( x\right)
}$\emph{\ and }$W^{k,p\left( x\right) }$\emph{.} Czechoslovak
Math. J. 41 (1991), 592--618.

\bibitem{} J. L. Lions: \emph{Queques methodes de resolution des
problemes aux limites non lineaires. }Dunod and Gauthier-Villars,
Paris (1969)

\bibitem{} G. de Marsily: \emph{Quantitative Hydrogeology. Groundwater
Hydrology for Engineers}. Academic Press, London (1986).

\bibitem{} J. Musielak:\emph{\ Orlicz Spaces and Modular Spaces}. Lecture
Notes in Mathematics, Vol.1034, Springer-Verlag, Berlin (1983).

\bibitem{} I. P. Natanson:\emph{\ Theory of Functions of a Real Variable}%
. GITTL, Moscow (1950).

\bibitem{} V. R\u{a}dulescu, D. Repov\v{s}: \emph{Partial Differential
Equations with Variable Exponents: Variational methods and
Quantitative Analysis}, CRC Press, Taylor \& Francis Group, Boca
Raton FL, (2015).

\bibitem{} K. Rajagopal, M. Ruzicka: \emph{Mathematical modeling of
electro-rheological fluids. } Contin. Mech. Thermodyn. 13, (2001),
59-78.

\bibitem{} M. Ruzicka: \emph{Electrorheological Fluids: Modeling and
Mathematical Theory. }In Lecture Notes in Mathematics, vol. 1748
Springer, Berlin (2000).

\bibitem{} U. Sert, K. N. Soltanov: \emph{On solvability of a class of
nonlinear elliptic type equation with variable exponent. }J. Appl. Anal. Comput. 7 (2017), no. 3, 1139-1160.

\bibitem{} K. N. Soltanov, J. Sprekels: \emph{Nonlinear equations in
non-reflexive Banach spaces and strongly nonlinear equations.
}Adv. Math. Sci. Appl. 9, no. 2 (1999), 939-972.

\bibitem{} K. N. Soltanov: \emph{Some imbedding theorems and nonlinear
differential equations. }Trans. Acad. Sci. Azerb. Ser. Phys.-Tech.
Math. Sci. 19 (1999), no. 5, Math. Mech. (2000), 125--146.
(Reviewer: H. Triebel)

\bibitem{} K. N. Soltanov: \emph{Some nonlinear equations of the
nonstable filtration type and embedding theorems. }Nonlinear Anal.
65 (2006), 2103-2134.

\bibitem{} K. N. Soltanov: \emph{Solvability of nonlinear equations with operators in the form of the sum of a pseudomonotone and a weakly compact operator}. Russian Acad. Sci. Dokl. Math., 45 (1992), no. 3, 676-681.

\bibitem{} E. Zeidler:\emph{\ Nonlinear functional analysis and its
applications. II/B. Nonlinear monotone operators.}
Springer-Verlag, New York (1990).

\bibitem{} V. V. Zhikov: \emph{On some variational problems}. Russian J.
Math. Phys., 5:1 (1997), 105-116.

\bibitem{} V. V. Zhikov: \emph{On the technique for passing to the limit
in nonlinear elliptic equations}. Functional Anal. and Its App.
Vol. 43, No. 2 (2009), 96-112.

\bibitem{} V. V. Zhikov: \emph{Averaging of functionals of the calculus
of variations and elasticity theory. }Math. USSR. Izv. 29 (1987),
33-36.

\end{thebibliography}
\end{document}